\documentclass{amsart}
 \pdfoutput=1
\usepackage[latin1]{inputenc}
\usepackage{amsmath}
\usepackage{amsthm}
\usepackage{amsfonts}
\usepackage{amssymb}
\usepackage{enumerate}
\usepackage{mathrsfs}
\usepackage{graphicx}
\usepackage{hyperref}
\usepackage{subfigure}
\usepackage{tikz}
\usetikzlibrary{matrix,arrows}

\newcommand\cub{\rho}
\newcommand\rad{\tilde{\sigma}}
\newcommand\radu{{\sigma}}

\newcommand\recta{\hat\sigma}
\newcommand\hur{h}
\newcommand\cnj{\tau}
\newcommand\gsiii{\Theta}
\newcommand{\comm}{\eta}

\newcommand\ba{{\mathbb A}}

\newcommand\bp{{\mathbb P}}
\newcommand\bs{\mathbb{S}}
\newcommand\bff{{\mathbb F}}
\newcommand\br{{\mathbb R}}
\newcommand\bc{{\mathbb C}}
\newcommand\bd{{\mathbb D}}

\newcommand\bn{{\mathbb N}}
\newcommand\bz{{\mathbb Z}}

\newcommand\bb{{\mathbb B}}

\newcommand\cB{{\mathcal B}}

\newcommand\cC{{\mathcal C}}

\newcommand\gs{{\sigma}}

\DeclareMathOperator\gal{Gal}

\newcommand\rightmap[1]{\smash{\mathop{\ \rightarrow\ }\limits^{#1}}} 

\newtheorem{thm}{Theorem}[section]

\newtheorem{prop}{Proposition}[section]

\newtheorem{lema}{Lemma}[section]

\theoremstyle{remark}
\newtheorem{obs}{Remark}[section]
\theoremstyle{definition}
\newtheorem{dfn}{Definition}[section]
\newtheorem{ntc}{Notation}[section]
\newtheorem{ejm}{Example}[section]
\newtheorem{paso}{Step}
\numberwithin{equation}{section}

\makeatletter
\let\c@lema\c@thm
\let\c@prop\c@thm
\let\c@conj\c@thm
\let\c@cor\c@thm
\let\c@obs\c@thm
\let\c@dfn\c@thm
\let\c@ntc\c@thm
\let\c@ejm\c@thm
\makeatother

\def\makeautorefname#1#2{\expandafter\def\csname#1autorefname\endcsname{#2}}

\makeautorefname{lema}{Lemma}%
\makeautorefname{obs}{Remark}%
\makeautorefname{thm}{Theorem}%
\makeautorefname{dfn}{Definition}%
\makeautorefname{ejm}{Example}%
\makeautorefname{subfigure}{Figure}%
\makeautorefname{prop}{Proposition}%

\newcommand\enet[1]{\renewcommand\theenumi{#1}
\renewcommand\labelenumi{\theenumi}}
\newcommand\eneti[1]{\renewcommand\theenumii{#1}
\renewcommand\labelenumii{\theenumii}}

\title{Kummer covers and braid monodromy}
\author[E. Artal]{Enrique Artal Bartolo}
\address{Departamento de Matem\'aticas, IUMA\\ 
Universidad de Zaragoza\\ 
C.~Pedro Cerbuna 12\\ 
50009 Zaragoza, Spain} 
\email{artal@unizar.es,jicogo@unizar.es} 

\author[J.I. Cogolludo]{Jos{\'e} Ignacio Cogolludo-Agust{\'i}n}

\author[J. Ortigas]{Jorge Ortigas-Galindo}
\address{Centro Universitario de la Defensa-IUMA\\ 
Academia General
Militar\\ 
Ctra. de Huesca s/n.\\ 
50090, Zaragoza, Spain} 
\email{jortigas@unizar.es}

\thanks{All authors are partially supported by
the Spanish Ministry of Education MTM2010-21740-C02-02.}  



\begin{document}

\begin{abstract}
In this work we describe a method to reconstruct the braid monodromy of the 
preimage of a curve by a Kummer cover. This method is interesting, since 
it combines two techniques, namely, the reconstruction of a highly non-generic braid
monodromy with a systematic method to go from a non-generic to a generic braid monodromy.
This ``\emph{generification}'' method is independent from Kummer covers and can be applied 
in more general circumstances since non-generic braid monodromies appear more naturally and 
are oftentimes much easier to compute. Explicit examples are computed using these techniques.
\end{abstract}

\maketitle

\section*{Introduction}
A Kummer cover is a map $\pi_n:\bp^2\to\bp^2$ given by $\pi_n([x:y:z]):=[x^n:y^n:z^n]$.
Kummer covers are a very useful tool in order to construct complicated
algebraic curves starting from simple ones. Since Kummer covers are
finite Galois covers of $\bp^2\setminus \{xyz=0\}$ with $\gal(\pi_n)\cong\bz/n\bz \times\bz/n\bz$, 
topological properties of the new curves can be obtained: Alexander polynomial, fundamental group, 
characteristic varieties and so on (see 
\cite{ea:jag,ac:98,ji:99,uludag:01,Hirano-construction,ji-kloosterman} 
for papers using these techniques). 

On the other hand, the braid monodromy of a plane projective curve is a powerful invariant 
that provides a way to compute the fundamental group of its complement and was originally 
described as a formalization of the Zariski-van Kampen method~\cite{zr:29,vk:33} (see 
\cite{lo:09,Dolgachev-Libgober-fundamental-group}
also~\cite{ji-Pau} and references therein for a detailed exposition on the subject). 
However, braid monodromy is a much more powerful invariant, since in fact it encodes the 
topology of the embedding of the curve, as well as the isomorphism problem for surfaces whose
branching locus over $\bp^2$ is a given 
curve (see~\cite{car:xx,kt:00,mz:81,Catanese-Wajnryb-3-cuspidal-quartic,
Loenne-moduli-spaces,Kulikov-generic-coverings,Auroux-Katzarkov-branched-coverings,
Cordovil-Fachada-braid-monodromy,Amram-Ciliberto-Miranda-Teicher-braid-monodromy} 
to name only a few).

In this work we focus our attention on a method to reconstruct the braid monodromy of the 
preimage of a curve by a Kummer cover. This method is interesting in and of itself, since 
it combines two techniques, namely, the reconstruction of a highly non-generic braid
monodromy with a systematic method to go from a non-generic to a generic braid monodromy.
This ``\emph{generification}'' method is independent from Kummer covers and can be applied 
in more general circumstances 
(see also~\cite{mztiv,acct:01,Amram-Teicher-degeneration}). 
The reason for this is that oftentimes, non-generic braid
monodromies are much easier to compute.

The importance of finding theoretical arguments to obtain braid monodromies lies in the fact
that any direct computation requires polynomial root approximation, whose exactness can only 
be justified in particularly simple cases, such as curves whose equation can be given in 
$\bz[x,y,z]$ (\cite{besmi,car:xx}), strongly real curves (\cite[p.~17]{act:08} and references
therein), or line arrangements 
(see~\cite{arvola,Ru2,cosu:97,sal:88,Cordovil-Fachada-braid-monodromy}).

The layout of the paper is as follows: after describing in \S\ref{sec-settings} the main objects 
to be used such as generic and non-generic braid monodromies, Kummer covers and extended braid
monodromies are described in \S\ref{sec-pencils} as the main tool to recover a non-generic braid 
monodromy of a Kummer cover from a braid monodromy of the base. Useful \emph{generification}
techniques are described in~\S\ref{sec-deformation}. Sections \S\ref{sec-cover-braids} 
and \S\ref{sec-trans} are more technical and describe useful systems of Artin generators for the
braids obtained after a Kummer cover and the properties and singularities of Kummer transforms.
Finally, a detailed account of numerous examples (some classical and some new) of the power of 
Kummer transforms is given in \S\ref{sec-ex}, where generic braid monodromies of smooth curves,
sextics with six cusps (on a conic and otherwise), sextics with nine cusps, Hesse, Ceva, and 
MacLane arrangements are provided.

\section{Settings: Braid monodromy}\label{sec-settings}

After the work of Zariski, braid monodromy was defined by Chisini~\cite{chi:33}
but it was necessary to wait for Moishezon~\cite{mz:81} in order to get
further applications of this invariant.

\subsection{Construction.}
\mbox{}

Let us fix a curve $\bar{\cC}\subset\bp^2$ of degree~$d$, a point $P_y\in\bp^2$ and a line
$\bar{L}_\infty$ such that $P_y\in \bar{L}_\infty$. We say that the curve is \emph{horizontal}
with respect to $P_y$ if it does not contain any line through~$P_y$; we assume
$\bar{\cC}$ to be horizontal. We consider a system of coordinates
$[X:Y:Z]$ such that $P_y:=[0:1:0]$ and $\bar{L}_\infty:=\{Z=0\}$. We identify
$\bc^2\equiv\bp^2\setminus\bar{L}_\infty$ with affine coordinates $(x,y)\equiv[x:y:1]$.

Let $F(x,y,z)=0$ be a reduced equation of $\bar{\cC}$, $k:=\deg_y F$
$$
F(x,y,z)=\sum_{j=0}^k \bar{a}_{d-j}(x,z) y^j,\quad \bar{a}_{d-k}(x,z)\neq 0,\quad
\bar{a}_j\text{ homogeneous of degree }j,
$$
normalized such that the coefficient of the term of higher degree of $\bar{a}_{d-k}(x,z)$ in $x$ is~$1$.
The fact that $\bar{\cC}$ is horizontal is equivalent to $\gcd(F,\bar{a}_{d-k})=1$.

The pencil of lines through $P_y$ is identified with $\bp^1\equiv\bc\cup\{\infty\}$, where $\infty$ 
corresponds with $\bar{L}_\infty$. Following the previous notation the lines in the pencil are denoted 
by~$\bar{L}_t:=\{X-t Z=0\}$. Let us restrict our attention to the affine part. Let $\cC:=\bar{\cC}\cap\bc^2$ 
and $L_t:=\bar{L}_t\cap\bc^2$; the line $L_t$ has equation $x=t$ while $\cC$ has equation $f(x,y)=0$, where
$$
f(x,y):=F(x,y,1)=\sum_{j=0}^k a_{d-j}(x) y^j,\quad
a_j(x):=\bar{a}_j(x,1).
$$
Let $\cB:=\{t\in\bc\mid\#(L_t\cap\cC)<k\}$; this is a finite set which contains the roots of $a_{d-k}(x)$
(if any) and the values~$t$ such that $L_t\not\pitchfork\cC$. The set~$\cB$ is the zero locus
of the product of $a_{d-k}(x)$ and the discriminant of $f(x,y)$ with respect to~$y$.

Let $\Sigma_k(\bc):=\{A\subset\bc\mid\#A=k\}$ be a configuration space of~$\bc$; for any
$A:=\{x_1,\dots,x_k\}\in\Sigma_k(\bc)$ the fundamental group
$\pi_1(\Sigma_k(\bc);A)=:\bb(x_1,\dots,x_k)$
is isomorphic to the braid group $\bb_k$ with the usual Artin presentation
\begin{equation}\label{eq-artinpres}
\bb_k:=
\left\langle
\sigma_1,\dots,\sigma_{k-1}\left| \ 
\underset{1<i+1<j<k}{[\sigma_i,\sigma_j]=1},\quad
\underset{1\leq i<k-1}{\sigma_i\cdot\sigma_{i+1}\cdot\sigma_i=\sigma_{i+1}\cdot\sigma_i\cdot\sigma_{i+1}}
\right.
\right\rangle. 
\end{equation}
For the next sections we need to describe a canonical identification between $\bb_k$ and $\bb(x_1,\dots,x_k)$;
the group $\pi_1(\Sigma_k(\bc);A)$ is identified with the homotopy classes of sets of arcs 
$\varphi_1,\dots,\varphi_k:[0,1]\to\bc$ starting
and ending in~$A$ and such that $\#\{\varphi_1(t),\dots,\varphi_k(t)\}=k$, $\forall t\in[0,1]$.
Let us order the points of $A$, say $x_1,\dots,x_k$ and consider a set $I$ of simple segments $A_i$, $1\leq i<k$,
such that $\partial A_i=\{x_i,x_{i+1}\}$, $A_i\cap A_{i+1}=\{x_{i+1}\}$ and the other intersections are empty;
such a collection~$I$ will be called a \emph{diagram system} for $(x_1,\dots,x_k)$.
Then we associate to $\sigma_i$ the braid which is constant for $x_1,\dots,x_{i-1},x_{i+2},\dots,x_k$ and
behaves in a small small closed (topological) disk $N(A_i)$
such that $A_i$ is a diameter as follows. The points $x_i$ and $x_{i+1}$
counterclockwise exchange  along $\partial N(A_i)$.
These generators may be understood as \emph{half-twists} around $A_i$. 

\begin{figure}[ht]
\begin{center}
\includegraphics[scale=1]{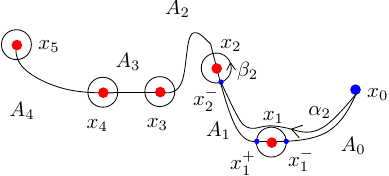}
\end{center}
\caption{Diagram system, $k=5$.}
\label{fig:diagrama}
\end{figure}

There is also a basis of the free group $\pi_1(\bc\setminus A;x_0)$ if one chooses a simple edge $A_0$ from
$x_0$ to $x_1$ intersecting $\bigcup_{i=1}^{k-1} A_i$ only at $x_1$. This basis $\mu_1,\dots,\mu_k$ is
obtained as follows: take small disks $\Delta_i$ centered at $x_i$ and assume that their intersection
with $A_{i-1}\cup A_i$ are diameters with ends $x_i^-,x_i^+$. Then $\mu_i$ is defined as follows:
take a path $\alpha_i$ from $x_0$ to $x_i^-$ running along $A_0\cup\dots\cup A_{i-1}$ outside
the interior of the disk $\Delta_j$ and goes counterclockwise along $\partial\Delta_j$ from $x_j^-$ to $x_j^+$,
$1\leq j\leq i$. Let $\beta_i$ be the closed path obtained by running counterclockwise along $\partial\Delta_j$
with base point $x_i^-$ and define $\mu_i:=\alpha_i\cdot\beta_i\cdot\alpha_i^{-1}$. 

The group $\bb_k$ acts geometrically on the group $\bff_k:=\pi_1(\bc\setminus A;x_0)$ as follows:
\begin{equation}\label{eq-action}
\mu_i^{\sigma_j}:=
\begin{cases}
\mu_{i+1}&\text{ if }j=i,\\
\mu_{i}\cdot\mu_{i-1}\cdot\mu_{i}^{-1}&\text{ if }j=i-1,\\
\mu_i&\text{ otherwise.}
\end{cases} 
\end{equation}

\begin{ntc}
\label{ntc-accion}
As usual, $a^b=b^{-1}ab$ is the conjugation of $a$ by $b$. Also, for short, we will write $b*a=bab^{-1}$.
\end{ntc}

\begin{figure}[ht]
\begin{center}
\includegraphics[scale=1]{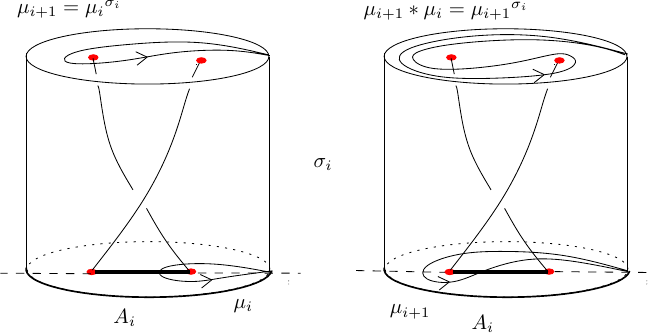}
\end{center}
\caption{Geometric version of the action of $\bb_k$ on  $\bff_k$.}
\label{fig:action}
\end{figure}

There are two important facts in these definitions; the element $\mu_\infty:=(\mu_k\cdot\ldots\cdot\mu_1)^{-1}$
is a meridian of the point at infinity and $\mu_\infty$ is a fixed point by the action of $\bb_k$.
We say that $(\mu_1,\dots,\mu_k)$ is an \emph{ordered geometric basis} of $\pi_1(\bc\setminus A;x_0)$.
As a general notation, if $G$ is a group and $\mathbf{x}:=(x_1,\dots,x_k)\in G^k$ we define
the \emph{pseudo-Coxeter element} of $\mathbf{x}$ as $c_\mathbf{x}:=x_k\cdot\ldots\cdot x_1$.

After this digression, note that $f$ defines a map $\tilde{f}:\bc\setminus\cB\to\Sigma_k(\bc)$.

\begin{dfn}\label{def-bm}
The \emph{braid monodromy} of the triple $(\bar\cC,P_y,\bar{L}_\infty)$ 
is the morphism
$$
\nabla:\pi_1(\bc\setminus\cB;t_0)\to\bb_{k},\quad t_0\in\bc\setminus\cB,
$$
defined by $\tilde{f}$ on the fundamental group. 
\end{dfn}

\begin{obs}\label{rem-bm}
Consider a geometric basis $(\gamma_1,\dots,\gamma_r)$ of $\pi_1(\bc\setminus\cB;t_0)$ and let $c_\infty$ 
be its pseudo-Coxeter element. Note that $\nabla$ is determined by 
$(\nabla(\gamma_1),\dots,\nabla(\gamma_r))\in\bb_k^r$ having as pseudo-Coxeter element~$\nabla(c_\infty)$.
\end{obs}

The braid monodromy measures the motions of the points of $\cC$ 
along the affine lines $L_t$ (identified with $\bc$).

There are a lot of choices in order to obtain an element
of $\bb_k^r$ from $(\bar{\cC},P_y,\bar{L}_\infty)$. 
It is not hard to check that these choices are given by the orbits of 
an action of $\bb_k\times\bb_r$ on $\bb_k^r$. The action of $\bb_k$ is given by simultaneous
conjugation. The action of $\bb_r$ is defined as follows; let $\hur_1,\dots,\hur_{r-1}$
an Artin system of generators of $\bb_r$. Then, if $(\tau_1,\dots,\tau_r)\in\bb_k^r$, then:
\begin{equation}\label{eq-hurwitz}
(\tau_1,\dots,\tau_r)^{\hur_i}:=
(\tau_1,\dots,\tau_{i-1},\tau_{i+1},\tau_{i+1}\cdot\tau_i\cdot\tau_{i+1}^{-1},\tau_{i+2},\dots,\tau_r);
\end{equation}
$\hur_i$ is called a \emph{Hurwitz move}.
In particular for a choice of $(\bar{\cC},P_y,\bar{L}_\infty)$ two objects are unique and well-defined: 
the conjugacy classes of the pseudo-Coxeter element and of the \emph{monodromy group}, i.e., the group
generated by $\nabla(\gamma_1),\dots,\nabla(\gamma_r)$. 

In light of the previous discussion, a braid monodromy $\nabla$ of a triple $(\bar{\cC},P_y,\bar{L}_\infty)$
will sometimes be considered as a morphism (see \autoref{def-bm}) or as a list of braids 
$(\nabla(\gamma_1),\dots,\nabla(\gamma_r))$, where $\gamma_1,\dots,\gamma_r$ is a geometric basis.

\subsection{Applications.}
\mbox{}

The first application of braid monodromy is the computation of the fundamental group,
see~\cite{acc:01a}.

Given $\nabla$ the braid monodromy of a triple $(\bar \cC,P_y,\bar L_\infty)$. Consider
$\bc^2:=\bp^2\setminus \bar L_\infty$, $\cC:=\bar \cC\cap \bc^2$, and $L_t$ ($t\in \cB$) 
the non-generic fibers of the pencil of lines through $P_y$ as above. 
One has the following result.

\begin{thm}\label{thm-zvk1} The group $\pi_1(\bc^2\setminus(\cC\cup\bigcup_{t\in\cB}L_t))$
has the following presentation:
$$
\left\langle
\mu_1,\dots,\mu_k,\tilde{\gamma}_1,\dots,\tilde{\gamma}_r
\left| \ 
\underset{1\leq i\leq r; 1\leq j\leq k}{\tilde{\gamma_i}^{-1}\cdot\mu_j\cdot\tilde{\gamma}_i=
\mu_j^{\nabla(\gamma_i)}} 
\right.
\right\rangle.
$$
\end{thm}

A triple $(\bar \cC,P_y,\bar L_\infty)$ (or simply a curve $\cC$) is said to be \emph{fully horizontal} 
if $\cC$ has no vertical asymptotes, i.e., if $b_{d-k}(x)=1$. In that case, if $\cB_0\subset\cB$ (maybe empty), 
corresponding to the generators $\{\gamma_{i_1},\dots,\gamma_{i_s}\}$, one has a new version
of Zariski-van Kampen Theorem, due to the fact that the above generators
$\tilde{\gamma}_j$ are meridians of the corresponding lines. 

\begin{thm}\label{thm-zvk2} The group $\pi_1(\bc^2\setminus(\cC\cup\bigcup_{t\in\cB_0}L_t))$
has the following presentation:
$$
\left\langle
\mu_1,\dots,\mu_k,\tilde{\gamma}_{i_1},\dots,\tilde{\gamma}_{i_s}
\left| \ 
\underset{1\leq j\leq s;\ 1\leq j\leq k}{\tilde{\gamma}_{i_j}^{-1}\cdot\mu_\ell\cdot\tilde{\gamma}_{i_j}=
\mu_\ell^{\nabla(\gamma_{i_j})}},\quad
\underset{i\neq i_1,\dots,i_s;\ 1\leq j< k}{\mu_j=
\mu_j^{\nabla(\gamma_{i})}}
\right.
\right\rangle.
$$
\end{thm}
There is also a statement like \autoref{thm-zvk2} for curves with vertical asymptotes but
it is more technical, see~\cite{car:xx} or~\cite{lo:09}.

Let us assume now that $P_y\notin\bar{\cC}$. In this case the triple $(\bar \cC,P_y,\bar L_\infty)$
is automatically fully horizontal and $d=k$.
In this case another version can be stated, which is closer to the original Zariski-van Kampen Theorem.

\begin{thm}\label{thm-zvk3} The group $\pi_1(\bp^2\setminus\bar{\cC})$
has the following presentation:
$$
\left\langle
\mu_1,\dots,\mu_d
\left| \
\underset{1\leq i\leq r;\ 1\leq j< d}{\mu_j=
\mu_j^{\nabla(\gamma_{i})}},\quad
\mu_d\cdot\ldots\cdot\mu_1=1
\right.
\right\rangle.
$$
\end{thm}

There is also a procedure to compute $\pi_1(\bp^2\setminus\bar{\cC})$ when $P_y\in\bar{\cC}$,
but the general formula is not so closed since it depends heavily on the local singularity 
$(\bar{\cC}\cup\bar{L}_\infty,P_y)$.

\subsection{Generic braid monodromies.}
\label{sec-generic}
\mbox{}

We finish this sequence of versions of Zariski-van Kampen Theorem with the generic case.
We assume $P_y\notin\bar{\cC}$, $\bar{L}_\infty\pitchfork\bar{\cC}$ (i.e they intersect at~$d$ distinct points),
and moreover for each $t\in\cB$ there is exactly one point $P_t\in\bar{L}_t\cap\bar{\cC}$ where
the intersection is not transversal and satisfies
\begin{equation}\label{eq-cond-gen}
(\bar{\cC}\cdot\bar{L}_t)_{P_t}=
\begin{cases}
2&\text{ if } (\bar{\cC},P_t)\text{ is smooth},\\
m_t&\text{ if } (\bar{\cC},P_t)\text{ is singular},
\end{cases}
\end{equation}
where $m_t$ is the multiplicity of the germ $(\bar{\cC},P_t)$.
In the singular case, it means that  $\bar{L}_t$ is not in the tangent cone of $(\bar{\cC},P_t)$
and, in the smooth case, that $P_t$ is not an inflection point.
In order to state the final form of Zariski-van Kampen Theorem, we need some notation.
The above conditions imply that for each $\gamma_i$, $1\leq i\leq r$, we can express
$\nabla(\gamma_i)=\eta_i\cdot\tau_i\cdot\eta_i^{-1}$, where $\eta_i,\tau_i\in\bb_d$
and $\tau_i$ is a positive word in $\Sigma_i:=\{\sigma_j\}_{j\in s_i}$, where 
$s_i\subset \{1,\dots,d-1\}$ and $\#s_i=(\bar{\cC}\cdot\bar{L}_t)_{P_t}-1$.

For each $i$, ($1\leq i\leq r$) the elements $\mu_j(i):=\mu_j^{(\eta_i^{-1})}$ represent another basis
of $\bff_d$. Let us denote by $\mu_j(i)^{\sigma_k(i)}$ the action~\eqref{eq-action}
written in the basis $\mu_1(i),\dots,\mu_d(i)$.

\begin{thm}\label{thm-zvk4} The group $\pi_1(\bp^2\setminus\bar{\cC})$
has the following presentation:
$$
\left\langle
\mu_1,\dots,\mu_d
\left| \
\underset{1\leq i< r;\ j\in s_i}{\mu_j(i)=
\mu_j(i)^{\tau_{i}(i)}},\quad
\mu_d\cdot\ldots\cdot\mu_1=1
\right.
\right\rangle.
$$
\end{thm}

In this case, the pseudo-Coxeter element of $(\nabla(\gamma_1),\dots,\nabla(\gamma_r))$
is $\Delta_d^2$, the positive generator of the central element of $\bb_d$.
This is why in the literature braid monodromy is also referred to as 
a \emph{factorization} of $\Delta_d^2$. The main point is that under these genericity conditions,
braid monodromy is an invariant of $\bar{\cC}$ independent of the (generic) choice
of $P_y$ and $\bar{L}_\infty$. Moreover, if we can connect two curves by a path of 
equisingular curves, then they share the braid monodromy. 

A more general type of genericity occurs simply when 
$P_y\notin\bar{\cC}$, $\bar{L}_\infty\pitchfork\bar{\cC}$. We will refer to such a triple as
\emph{generic at infinity}. In this case \autoref{thm-zvk4} is still true, but the sets $s_i$ have to 
be replaced by a finite number of disjoint subsets $s_{i,1},\dots,s_{i,\kappa_i}\subset \{1,\dots,d-1\}$.

An interesting case of genericity at infinity occurs when condition~\eqref{eq-cond-gen} holds, but each 
fiber $L_t$, $t\in \cB$ is allowed to have more than one ramification point. We refer to this case 
as \emph{local genericity}. From a locally generic braid monodromy, a generic braid monodromy might 
easily be reconstructed (see \autoref{prop-gen3}).

\section{Pencils of lines and Kummer covers}\label{sec-pencils}

Following notation from~\S\ref{sec-settings}, 
let us fix a curve $\bar{\cC}\subset\bp^2$ of degree~$k$.
We assume that $P_y:=[0:1:0]\notin\bar{\cC}$.
Let $\bar{\cC}_n:=\pi_n^{-1}(\bar{\cC})$ be its Kummer transform.
Denote by~$\bar{L}_t:=\{X-t Z=0\}$ the pencil of lines through~$P_y$.

\begin{obs}
Most of the results of this section also hold
if $P_y\in \bar{\cC}$ and the tangent cone of $(\bar{\cC},P_y)$ is $Z=0$
(or alternatively, if a change of variables can be performed to match these conditions).
\end{obs}

Let $\bar{L}_Y:=\{Y=0\}\subset\bp^2$ and let $L_Y$ be its affine part.
We consider the set of lines 
through $P_y$ 
which are not transversal to $\bar{\cC}\cup\bar{L}_Y$
and to which we include the lines $\bar{L}_0$ and $\bar{L}_\infty$ if necessary
(due to their special relation with $\pi_n$). Let 
$$
B:=\{t\in\bp^1\mid \bar{L}_t\not\pitchfork(\bar{\cC}\cup\bar{L}_Y)\},\quad
B^*:=B\cap\bc^*,\quad 
\tilde{B}:=B^*\cup\{0\},\quad
B^\infty:=B\cup\{\infty\}.
$$
The braid monodromy of $\bar{\cC}\cup\bar{L}_Y$ with respect to $P_y$
and $\bar{L}_\infty$ is the morphism
$$
\nabla:\pi_1(\bc\setminus B;t_0)\to\bb_{k+1},\quad t_0\in\bc\setminus\tilde{B}.
$$
Note that the image of $\nabla$ is contained in $\bb_{k,1}$,
the subgroup of $\bb_{k+1}$ given by all the braids whose 
associated permutation fixes a given point
(see~\S\ref{sec-cover-braids} for a description of this group).

\begin{dfn}\label{dfn-extended-bm}
The map 
\begin{equation}\label{eq-gbm}
\tilde{\nabla}:\pi_1(\bc\setminus\tilde{B};t_0)\to\bb_{k,1}
\end{equation}
is called the \emph{extended braid monodromy}
of the triple $(\bar \cC,P_y,L_\infty)$ with respect to $L_Y$.
\end{dfn}

Note that if $0\notin B$ then the image by $\tilde{\nabla}$ of a meridian around~$0$
is the trivial braid.

We define the analogous objects for $\bar{\cC}_n$: $B_n,B^*_n,\tilde{B}_n,B^\infty_n,\nabla_n,\tilde{\nabla}_n$
for $t_{0,n}$, an $n$-th root of $t_0$. The following Lemmas are easy consequences
of the properties of a cover.

\begin{lema}\label{lema-aux1}
Let $\phi_n:\bc\to\bc$ be the map defined by $\phi_n(t):=t^n$. Then,
$B^*_n=\phi_n^{-1}(B^*)$ and $\tilde{B}_n=\phi_n^{-1}(\tilde{B})$.
\end{lema}

\begin{lema}\label{lema-aux2}
The map $\pi_n$ induces degree $n$ coverings $\varphi_n:L_t\mapsto L_{t^n}$.
Moreover, the preimage of a line $L_s$ by $\pi_n$ 
is the disjoint union of $L_t$, $t^n=s$, each one inducing a covering $\varphi_n$ as above. 
\end{lema}

As a consequence of these lemmas, the braid monodromy $\tilde{\nabla}_n$ can be thought as follows:
consider the loops $\gamma$ in $\bc\setminus\tilde{B}$ which can be lifted as loops by $\phi_n$
and replace the braid $\nabla(\gamma)$, where 
$k$~points move in $\bc^*$, by the braid obtained by the constant string $0$ and the $n$-roots 
of these $k$ points. This construction defines a map
$$
\tilde{\cub}_{n,k}:\bb_{k,1}\to\bb_{n k,1}
$$
which is described in \S\ref{sec-cover-braids}.

Let us summarize the results. If $r:=\# B^*$, we denote by $\bff_{r+1}$ the free group
$\pi_1(\bc\setminus\tilde{B};t_0)$. Analogously, we denote by $\bff_{n r+1}$
the free group $\pi_1(\bc\setminus\tilde{B}_n;t_{0,n})$. Using covering theory, the map
$(\phi_n)_*$ fits in a short exact sequence
$$
0\to\bff_{n r+1}\to\bff_{r+1}\to\bz/n\bz\to 0.
$$
The following diagram holds:
\begin{equation}\label{eq-diag-n}
\begin{tikzpicture}[description/.style={fill=white,inner sep=2pt},baseline=(current bounding box.center)]
\matrix (m) [matrix of math nodes, row sep=3em,
column sep=2.5em, text height=1.5ex, text depth=0.25ex]
{ \bff_{r+1}&  \bb_{k,1} \\
\bff_{n r+1}& \bb_{n k,1}  \\ };
\path[->,>=angle 90](m-1-1) edge node[auto] {$  \tilde{\nabla} $} (m-1-2);
\path[->,>=angle 90](m-2-1) edge node[auto] {$  \tilde{\nabla}_n $} (m-2-2);
\path[<-right hook,>=angle 90](m-1-1)  edge node[auto] {}  (m-2-1);
\path[right hook->,>=angle 90](m-1-2)  edge node[auto] {}  (m-2-2);
\end{tikzpicture}
\end{equation}
Using this diagram, one can recover geometric bases and relate the central elements of the 
corresponding braid groups. The following makes this construction useful.

Summarizing, let $\tilde \nabla:=(\tau_1,\dots,\tau_{r+1})\in\bb_{k,1}^{r+1}$ be an extended braid 
monodromy for $(\bar{\cC},P_y,\bar{L}_\infty)$ with respect to $L_Y$, and $\tau_{r+1}$ is the braid 
corresponding to $\bar{L}_0$. Consider $\pi_n(X:Y:Z)=[X^n:Y^n:Z^n]$ the $n$-th Kummer cover of $\bp^2$ 
and denote by $\bar \cC_n$ the preimage of $\bar \cC$ by $\pi_n$. The braid monodromy $\tilde \nabla$ produces 
a list of braids $\tilde \nabla_n\in \bb_{nk,1}^{nr+1}$ as described in diagram{\rm~\eqref{eq-diag-n}}. 
One has the following

\begin{prop}\label{prop-kummer-bm}
The element $\tilde\nabla_n$ described above is an extended braid monodromy for the triple 
$\pi_n^{-1}(\bar \cC,P_y,\bar L_\infty)=(\bar \cC_n,P_y,\bar L_\infty)$ with respect to~$\pi_n^{-1}(\bar L_Y)=\bar L_Y$.
\end{prop}

\begin{proof}
It is an immediate consequence of the previous construction and Lemmas~\ref{lema-aux1} and \ref{lema-aux2}.
\end{proof}

\begin{obs}
Note that the braid monodromy obtained from $\tilde\nabla_n$ via the forgetful map $\bb_{nk,1}\to \bb_{nk}$
is (for the interesting cases) a highly non-generic braid monodromy.
\end{obs}

\begin{lema}\label{lema-geometric-basis-cover}
Let $(\gamma_1,\dots,\gamma_{r+1})$ be a geometric basis of $\bff_{r+1}$
such that $\gamma_{r+1}$ is a meridian of~$0\in \tilde{B}$. Then, the ordered 
list
$$(v_0,...,v_j,...,v_{n-1},\gamma_{r+1}^{n}),$$
where $v_j:=(\gamma_1^{\gamma_{r+1}^{j}},\dots,\gamma_r^{\gamma_{r+1}^{j}})$
(see Notation{\rm~\ref{ntc-accion}}),
forms a geometric basis of $\bff_{n r+1}$.
\end{lema}

\begin{lema}\label{lema-image-central}
Let $\Delta_{k+1}^2\in\bb_{k,1}$ be the positive generator of the center of $\bb_{k+1}$.
Then $\Delta_{k+1}^{2 n}=\Delta_{n k+1}^{2}$ via the inclusion $\bb_{k,1}\hookrightarrow\bb_{n k,1}$.
\end{lema}

The purpose of the upcoming sections will be to describe an effective way to obtain a generic 
braid monodromy for $\bar\cC_n$ from the given computation of $\tilde{\nabla}_n$. This is an outline 
of the general strategy, which will be developed in what follows:

\begin{itemize}
\item If $\bar{L}_Y\nsubseteq\bar{\cC}_n$, we compose $\tilde{\nabla}_n$ with the natural
map $\bb_{n k,1}\to\bb_{n k}$ obtained by forgetting the constant string~$0$.
\item If $0\notin B_n$ then the map factorizes through $\bff_{n r+1}\twoheadrightarrow\bff_{n r}$
whose kernel is generated by the meridians around~$0$.
\item If $\infty\in B_n$ we change the line at infinity; the image $\alpha$ of a meridian at infinity
is obtained as follows. Let $\tau$ be the pseudo-Coxeter of the braid monodromy; 
then $\alpha:=\Delta^{2}\tau^{-1}$, where $\Delta^{2}$ is the positive generator of the
center of the braid group.
\item Slightly move the projection point $P_y$ to obtain a generic projection.
\end{itemize}

\section{The motion of $n$-roots}\label{sec-cover-braids}

As was introduced in \S\ref{sec-pencils}, let $\bb_{k,1}$ be the subgroup of $\bb_{k+1}$ 
given by all the braids whose associated permutation fixes a given point in $\bc$, say~$0$.

We identify $\bb_{k+1}$ with the group of braids with ends in $k+1$
non-negative real points $x_1>\dots>x_k>x_{k+1}=0$ and we
consider the Artin generators $\sigma_1,\dots,\sigma_k$ in $\bb_{k+1}$
which are geometrically associated with the diagram system obtained from
the paths joining these points in the real line, see~\S\ref{sec-settings}.
Recall that any choice of a \emph{diagram system}~$I$ of $k+1$ points in $\bc$ induces an Artin 
system of generators~\eqref{eq-artinpres} of~$\bb_{k+1}$.
Note that the only half-twist $\sigma_i$ that moves $0$ is $\sigma_k$.
The following result is well known.

\begin{lema}
The group $\bb_{k,1}$ is generated by $\sigma_1,\dots,\sigma_{k-1},\sigma_k^2$
(as a subgroup of $\bb_{k+1}$) and it is naturally isomorphic to the group of braids of 
$k$ strands in $\bc^*$.
\end{lema}

Let us consider the map $\cub_n:\bc^*\to\bc^*$, given by $\cub_n(z):=z^n$.
Note that this map induces a morphism $\tilde \cub_n:\bb_{k,1}\to \bb_{n k,1}$ via the multivalued
function $\cub_n^{-1}$. The goal of this section is to give explicit formul{\ae} for
this morphism. In order to do so, one needs to explicitly choose systems of generators for 
$\bb_{k+1}$ and $\bb_{n k+1}$.

Consider the diagram system $I$ for $\bb_{k,1}$ described above. The image of the half twist $\sigma_i$,
associated with $A_i$, $1\leq i<k$, 
is a product of half-twists associated with the pairwise disjoint arcs $A_{i,1},...,A_{i,n}$
such that $\cub_n^{-1}(A_i)=A_{i,1}\cup\dots\cup A_{i,n}$. The image of $\sigma_k^2$ is more complicated
since  $A_{k,1},...,A_{k,n}$ intersect at~$0$.
However, note that the natural system of arcs $A_{i,j}$ does not produce a diagram system for $\bb_{n k+1}$. 
In order to solve this situation one needs to define diagram systems for $\bb_{n k,1}$ that allow to 
describe the arcs $A_{i,j}$ naturally. For different purposes, different diagram systems might be more 
suited. Here we will concentrate our attention in two particular diagram systems: the circular and the
radial.

\subsection{Circular diagram systems}
\mbox{}

Let us denote $d:=nk$. We identify $\bb_{d+1}$ with 
$\bb(\{\xi_n^j t_i\mid 1\leq i\leq k, 1\leq j\leq n\}\cup\{0\})$, where $t_i$ is the non-negative
$n$-th root of $x_i\in \br_{\geq 0}$ and $\xi_n:=\exp(\frac{2\pi\sqrt{-1}}{n})$. Consider the following arcs:
\begin{itemize}
\item $c_{i,j}$ is the counterclockwise arc joining $t_i\xi_n^{j-1}$ to $t_i\xi_n^{j}$ in the circle
centered at~$0$, $1\leq j<n$;
\item $c_{i,n}$ is the segment  joining $t_i\xi_n^{n-1}$ 
and $t_{i+1}$.
\end{itemize}
These arcs and segments are illustrated in \autoref{fig:dibujo1}.
The list of arcs $\{c_{i,j}\}_{i,j}$ with a left-lexicographic order 
produces a diagram system called the \emph{circular diagram system}.
The half twist produced by an arc $c_{i,j}$ will be denoted by $\sigma_{i,j}\in \bb_{n k,1}$.
The  half-twists associated with~$A_{i,j}$ will be denoted by
$\alpha_{i,j}$.

Figure~\ref{fig:dibujo2_1} is obtained from Figure~\ref{fig:dibujo1} after unwinding the circular diagram 
system into a straight line to get the usual projections. This will help visualize the rewriting of the half-twists
$\alpha_{i,j}$ in terms of the $\sigma_{i,j}$.

\begin{figure}[ht]
\centering
\subfigure[Circular diagram system of $\bb_{n k+1}$ for $k=n=3$]{\makebox[.45\textwidth]{
\includegraphics{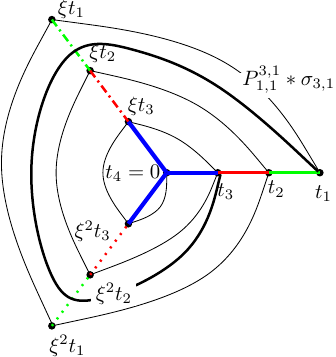}}
\label{fig:dibujo1}
}
\hfill
\subfigure[A \emph{straightened} view of \autoref{fig:dibujo1}.]{
\includegraphics{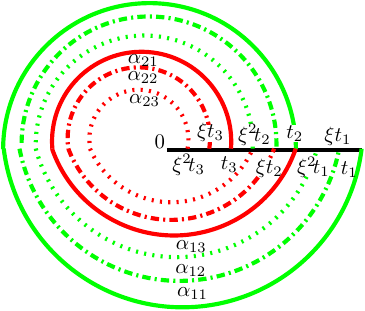}
\label{fig:dibujo2_1}
}
\label{fig1}
\caption{}
\end{figure}
Consider $(i_1,j_1)<(i_2,j_2)$, where $<_\ell$ represents the left-lexicographic order (that is, either $i_1<i_2$
or $i_1=i_2$ and $j_1<j_2$). We will define 
\begin{equation}
\label{eq-P}
\array{c}
P_{i_1,j_1}^{i_2,j_2}:=\sigma_{i_1,j_1}\cdot\ldots\cdot \sigma_{i_2,j_2}=
\displaystyle
\prod_{(i_1,j_1)\leq_\ell v\leq_\ell (i_2,j_2)} \sigma_v\\
\bar{P}^{i_2,j_2}_{i_1,j_1}:=\sigma_{i_1,j_1}^{-1}\cdot\ldots\cdot \sigma_{i_2,j_2}^{-1}=
\displaystyle
\prod_{(i_1,j_1)\leq_\ell v\leq_\ell (i_2,j_2)} \sigma_v^{-1}\\
M_{i_1,j_1}^{i_2,j_2}:=\sigma_{i_2,j_2}\cdot\ldots\cdot \sigma_{i_1,j_1}=
\displaystyle
\prod_{(i_2,j_2)\geq_\ell v\geq_\ell (i_1,j_1)} \sigma_v\\
\bar{M}^{i_2,j_2}_{i_1,j_1}:=\sigma_{i_2,j_2}^{-1}\cdot\ldots\cdot \sigma_{i_1,j_1}^{-1}=
\displaystyle
\prod_{(i_2,j_2)\geq_\ell v\geq_\ell (i_1,j_1)} \sigma_v^{-1}.
\endarray
\end{equation}

\begin{obs}
\label{rem-P}
Note that $P_{i_1,j_1}^{i_2,j_2}=(\bar{M}_{i_1,j_1}^{i_2,j_2})^{-1}$, 
$(\bar{P}_{i_1,j_1}^{i_2,j_2})^{-1}=M_{i_1,j_1}^{i_2,j_2}$. 
Also, note that 
$$P_{i_1,j_1}^{i_2,j_2}*\sigma_{i_2,j_2}=\bar{M}_{i_1,j_1}^{i_2,j_2}*\sigma_{i_1,j_1}\quad
(\text{resp. } \bar{P}_{i_1,j_1}^{i_2,j_2}*\sigma_{i_2,j_2}^{-1}=M_{i_1,j_1}^{i_2,j_2}*\sigma_{i_1,j_1}^{-1}) 
$$
represents a half-twist (resp. a negative half-twist) interchanging $t_{i_1}\xi_n^{j_1-1}$ 
and $t_{i_2}\xi_n^{j_2+1}$ along a spiral arc. 
The arc corresponding to
$P_{1,1}^{3,1}*\sigma_{3,1}=\bar{M}_{1,1}^{3,1}*\sigma_{1,1}$ 
(or $\bar{P}_{1,1}^{3,1}*\sigma_{3,1}^{-1}=M_{1,1}^{3,1}*\sigma_{1,1}^{-1}$) is shown
in~\autoref{fig:dibujo1}.
\end{obs}

\begin{lema}
\label{lema-PM}
Under the above conditions and for any $(i_1,j_1)<(i,j)<(i_2,j_2)$ one has
\begin{enumerate}
\enet{\rm(\arabic{enumi})}
 \item \label{lema-PM1}
$P_{i_1,j_1}^{i_2,j_2}*\sigma_{i_2,j_2}=
\left(\bar{M}_{(i,j)^+}^{i_2,j_2}\cdot P_{i_1,j_1}^{i,j}\right)*\sigma_{i,j}$,
where $(i,j)^+$ denotes the element following $(i,j)$ in the left-lexicographic order.
 \item \label{lema-PM2}
$\bar P_{i_1,j_1}^{i_2,j_2}*\sigma_{i_2,j_2}=
\left(M_{(i,j)^+}^{i_2,j_2}\cdot \bar P_{i_1,j_1}^{i,j}\right)*\sigma_{i,j}$.
 \item \label{lema-PM3}
$\alpha_{i,j}\!=\!\left(P_{i+1,j}^{k,n}\bar{P}_{i,j}^{k,n}\right)\!*\!\sigma_{k,n}\!=\!
\left(P_{i+1,j}^{k,n}\sigma_{k,n}^2 M_{(i_2,j_2)^+}^{k,n-1} \bar{P}_{i,j}^{i_2,j_2}\right)\!*\!\sigma_{i_2,j_2}\!=\!
\left(P_{i+1,j}^{k,n-1} \sigma_{k,n}^2 M_{i,j}^{k,n-1}\right)\!*\!\sigma_{i,j}
$.
\end{enumerate}
\end{lema}

\begin{proof}
Property~\eqref{lema-PM1} is immediate by induction since 
$$
\array{c}
P_{i_1,j_1}^{i_2,j_2}*\sigma_{i_2,j_2}=
\left(P_{i_1,j_1}^{(i_2,j_2)-2}\cdot \sigma_{(i_2,j_2)^-}\right)*\sigma_{i_2,j_2}=\\
\left(P_{i_1,j_1}^{(i_2,j_2)-2}\cdot \sigma_{i_2,j_2}^{-1}\right)*\sigma_{(i_2,j_2)^-}=
\left(\sigma_{i_2,j_2}^{-1}\cdot P_{i_1,j_1}^{(i_2,j_2)^-}\right)*\sigma_{(i_2,j_2)^-},
\endarray
$$
where $(i_2,j_2)^-$ denotes the element immediately smaller than $(i_2,j_2)$ and $(i_2,j_2)-2$ 
denotes its second predecessor~$\left((i_2,j_2)^-\right)^-$. Property~\eqref{lema-PM2} is analogous.

In order to show Property~\eqref{lema-PM3} note that $\alpha_{i,j}$ is a half-twist that exchanges
$t_i\xi_n^{j-1}$ and $t_{i+1}\xi_n^{j-1}$. Therefore, according to \autoref{rem-P} one has
$$
\alpha_{i,j}=
\left(P_{i+1,j}^{k,n}*\sigma_{k,n}\right)\cdot 
\left(\bar{P}_{i,j}^{k,n}*\sigma_{k,n}\right)\cdot 
\left(P_{i+1,j}^{k,n}*\sigma_{k,n}^{-1}\right).
$$
A similar induction argument shows that
$$\left(P_{i+1,j}^{k,n}*\sigma_{k,n}\right)\cdot 
\left(\bar{P}_{i,j}^{k,n}*\sigma_{k,n}\right)\cdot 
\left(P_{i+1,j}^{k,n}*\sigma_{k,n}^{-1}\right)=
\left(P_{i+1,j}^{k,n}\bar{P}_{i,j}^{k,n}\right)*\sigma_{k,n}.$$
The other equalities are a consequence of \autoref{rem-P} and 
Properties~\eqref{lema-PM1} and~\eqref{lema-PM2}. 
\end{proof}

\begin{prop}\label{prop-motion1}
The map $\tilde \cub_n:\bb_{k,1}\to \bb_{d,1}$ is given by
\begin{alignat*}{5}
\sigma_i&\mapsto\quad&\prod_{j=1}^{n} \alpha_{i,j},&\quad&i\in 1,\ldots, k-1\\
\sigma_{k}^{2}&\mapsto&\sigma_{k,n}^2\sigma_{k,n-1}\sigma_{k,n-2}\cdot\ldots\cdot\sigma_{k,1}&&
\end{alignat*}
\end{prop}

\begin{proof}
As mentioned at the beginning of this section, if $i\neq k$, the preimage of $A_i$ by $\cub_n$ 
is a disjoint union of segments $A_{i,1},\dots,A_{i,n}$. Therefore the image of the half-twist
$\sigma_i$ associated with $A_i$ is the product of the $n$ half-twists 
$\alpha_{i,1}\cdot\ldots\cdot \alpha_{i,n}$.
\begin{figure}[ht]
\begin{center}
\includegraphics[scale=1]{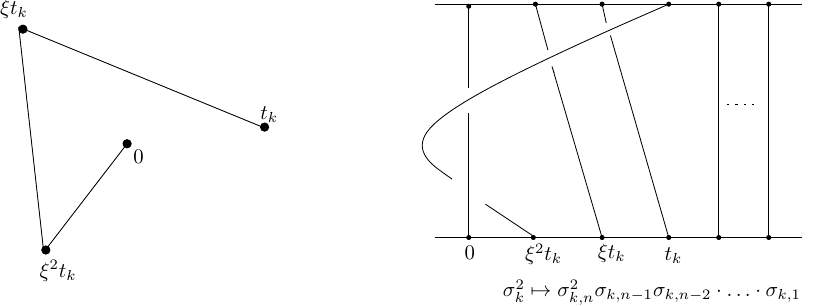}
\end{center}
\caption{Image of $\sigma_k^2$, $n=3$.}
\label{fig:sk2}
\end{figure}
Note that $\sigma_k^2$ corresponds to a full turn around $0$, its preimage by $\cub_n$ being
a counterclockwise rotation of angle~$\frac{2\pi}{n}$ of the points $\xi^j_n t_k$. This is 
nothing but $\sigma_{k,n}^2\cdot \sigma_{k,n-1}\cdot\ldots\cdot\sigma_{k,1}$
(see Figure~\ref{fig:sk2}).
\end{proof}

\begin{ejm}\label{ex-0}
Let us consider the case $k=1$ and the composition of $\tilde \cub_n$
with the natural projection of $\bb_{n,1}\twoheadrightarrow\bb_k$ given by forgetting
the string~$0$. Then we have a map $\hat{\cub}_n:\bb_{1,1}\to\bb_n$ such that
$\hat{\cub}_n(\sigma_1^2)=\sigma_{n-1}\sigma_{n-2}\cdot\ldots\cdot\sigma_{1}$.
\end{ejm}

\begin{ejm}\label{ex-1}
Next consider the case $k=2$ and the map $\hat{\cub}_n:\bb_{2,1}\to\bb_{2 n}$ as in~\autoref{ex-0}. 
Then, according to Lemma~\ref{lema-PM}\eqref{lema-PM3}
$$
\hat{\cub}_n(\sigma_1)=\prod_{j=1}^n
\left(P_{2,j}^{2,n-1} M_{1,j}^{2,n-1}\right)*\sigma_{1,j}=
\prod_{j=1}^n\left( 
\sigma_{2,j}\cdot\ldots\cdot \sigma_{2,n-1}^2 \cdot\ldots\cdot \sigma_{1,j-1}
\right)*\sigma_{1,j}
$$
and
$$
\hat{\cub}_n(\sigma_2^2)=\sigma_{2,n-1}\cdot\ldots\cdot\sigma_{2,1}.
$$
\end{ejm}

\subsection{Radial diagram systems.}\label{sec-radial}
\mbox{}

For some applications it is more useful to consider a different diagram system which will
be called \emph{radial}, see \autoref{fig:2bis}. 

\begin{figure}[ht]
\centering
\subfigure[Radial diagram system of $\bb_{nk+1}$ for $k=n=3$.]{\makebox[.45\textwidth]{
\includegraphics[scale=1]{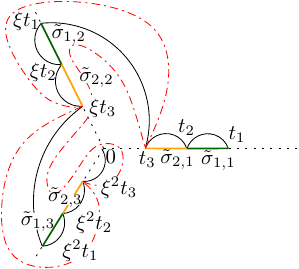}}
\label{fig:2bis}
}
\hfill
\subfigure[A \emph{straightened} view of Figure~\ref{fig:2bis}.]{
\includegraphics[scale=1]{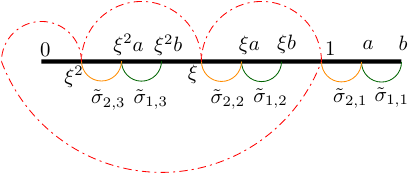}
\label{fig:dibujo2_2}
}
\label{fig2}
\caption{}
\end{figure}

We call the Artin generators for this system $\rad_{i,j}$, $1\leq i\leq k$, $1\leq j\leq n$. 
According to this choice, these are the half-twists $\alpha_{i,j}$ associated with the segments 
$A_{i,j}$ for $1\leq i\leq k-1$, that is, 
\begin{equation}
\label{eq-alphas}
\rad_{i,j}=\alpha_{i,j}.
\end{equation}
However, the element $\rad_{k,j}$ corresponds to an arc joining $t_k\xi_n^j$ with $t_1\xi_n^{j+1}$
as shown in Figure~\ref{fig:2bis}. Denote
$\beta_{j}:=\rad_{k-1,j}\cdot\ldots\cdot\rad_{1,j}$
Then we have 
\begin{equation}
\label{eq-betas} 
\sigma_{k,j}=\beta_{j+1}\rad_{k,j}\beta_{j+1}^{-1}\quad (1\leq j\leq n-1)
\end{equation}
Moreover, $\sigma_{k,n}=\rad_{k,n}$.
The following result holds.
\begin{prop}\label{prop-motion2}
The map  $\tilde \cub_n:\bb_{k,1}\to \bb_{d,1}$ is given by
\begin{alignat*}{5}
\sigma_i&\mapsto\quad&\prod_{j=1}^{n} \rad_{i,j},&\quad&1\leq i\leq k-1\\
\sigma_{k}^{2}&\mapsto&
\rad_{k,n}^2\prod_{(k-1,n)\geq_r v\geq_r (k,1)} \rad_{v}\cdot \prod_{j=2}^{n} \beta_{j}^{-1},&&
\end{alignat*}
where $\geq_r$ represents the right-lexicographic order.
\end{prop}

\begin{proof}
The result follows from \autoref{prop-motion1}. The formula for the image of $\sigma_i$ 
is a direct consequence of~\eqref{eq-alphas}. The formula for the image of $\sigma_k^2$ can be obtained 
using~\eqref{eq-betas} and the fact that $\beta_j$ commutes with $\beta_{j'}$ and with $\rad_{i,j'}$ for
$j'\leq j$ and $(i,j')\neq (k,j-1)$ (see \autoref{fig:trenza2}).
\end{proof}

\begin{figure}[ht]
\begin{center}
\includegraphics[scale=1]{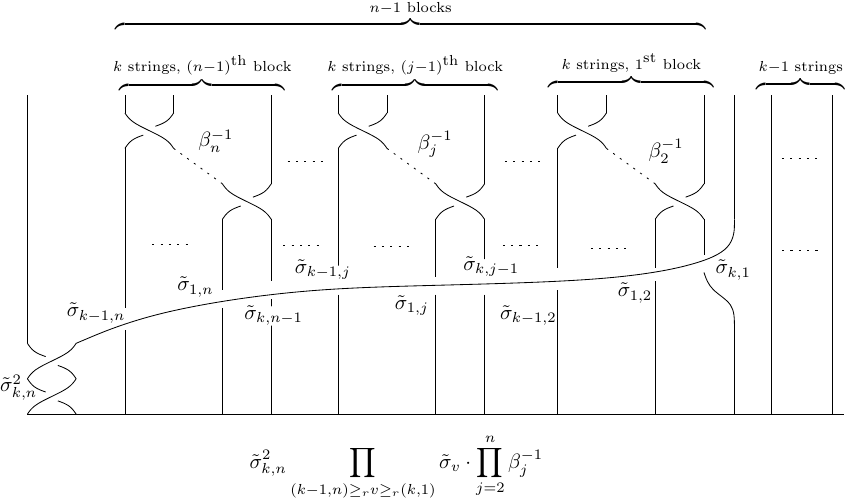}
\end{center}
\caption{Image of $\sigma_k^2$ by \autoref{prop-motion2}.}
\label{fig:trenza2}
\end{figure}

\begin{ejm}\label{ex-cover2}
Let us consider the case $n=2$ and $\hat{\cub}_2:\bb_{k,1}\to\bb_{2 k}$ the composition of $\tilde{\cub}$ 
with the forgetful map as in~\autoref{ex-0}. According to \autoref{prop-motion2} one has
\begin{alignat*}{5}
\sigma_i&\mapsto\quad&\rad_{i,1}\rad_{i,2},&\quad&1\leq i\leq k-1\\
\sigma_{k}^{2}&\mapsto&
(\rad_{k-1,2}\cdot\ldots\cdot \rad_{1,2})* \rad_{k,1}
\end{alignat*}
In fact, a more convenient description of these maps can be given for another choice
of generators $\recta_i$ coming from the natural diagram system on the real line
joining the preimages of $x_k$ (which they all lie on the real real line since $n=2$).
$$
\begin{matrix}
\cub'_2:&\bb_{k,1}&\to&\bb_{2 k+1}&\hat{\cub}'_2:&\bb_{k,1}&\to&\bb_{2 k}\\
&\sigma_i&\mapsto&\recta_i\recta_{2k-i+1}&&\sigma_i&\mapsto&\recta_i\recta_{2k-i}\\
&\sigma_k^2&\mapsto&\recta_{k+1}\recta_{k}\recta_{k+1}&&\sigma_k^2&\mapsto&\recta_k.
\end{matrix}
$$
\end{ejm}

\section{Deformations of braid monodromies}\label{sec-deformation}

We are mostly interested in computing generic braid monodromies of curves
(see \S\ref{sec-generic}).
However, it is usually easier (and oftentimes more natural)
to compute non-generic braid monodromies. For instance, those obtained
after Kummer covers (even if the starting monodromy was generic) are typically non generic. 
In this section we show how to deal with non-generic braid monodromies.

\subsection{From fully horizontal to generic at infinity.}\label{sec-fh-gi}
\mbox{}

Let us consider the braid monodromy of a fully horizontal triple~$(\bar{\cC},P,\bar{L})$,
where $P\notin\bar{\cC}$ (e.g. the hypothesis in \autoref{thm-zvk3}), where $\bar{L}$ and 
$\bar \cC$ are not necessarily transversal.

\begin{prop}\label{prop-gen1}
Assume $\bar{L}\not\pitchfork\bar{\cC}$, let $(\tau_1,\dots,\tau_r)\in\bb_d^r$ be a braid monodromy 
factorization for $(\bar{\cC},P,\bar{L})$, and consider a line $\bar{L}'$ such that $P\in\bar{L}'$ 
and $\bar{L}'\pitchfork\bar{\cC}$.
Then $(\tau_1,\dots,\tau_r,\Delta_d^2(\tau_r\cdot\ldots\cdot\tau_1)^{-1})$ 
is a braid monodromy factorization for $(\bar{\cC},P,\bar{L}')$.
\end{prop}

\begin{proof}
Note that $\cB'=\cB \cup \{t_L\}$, where $\cB$ (resp. $\cB'$) is the ramification set for 
$(\bar{\cC},P,\bar{L})$ (resp. for $(\bar{\cC},P,\bar{L}')$).
Since the line $\bar{L}'$ is transversal, the pseudo-Coxeter of its braid monodromy factorization
is~$\Delta_d^2$ and the result follows.
\end{proof}

This will be a common situation when studying Kummer covers for curves which are not transversal
to the axes. In that case we can be more explicit.

Let $T:=(\tau_1,\dots,\tau_r,\tau_{r+1})\in\bb_{k,1}^{r+1}$ be an extended braid monodromy factorization 
for $(\bar{\cC},P_y=[0:1:0],\bar{L}_\infty)$ with respect to $\bar{L}_Y$. Assume that $\bar{L}_0=\{X=0\}$ 
is not transversal to $\bar \cC$, and that $\tau_{r+1}$ is the braid corresponding to $\bar{L}_0$. 
Let us denote by $c_T$ its pseudo-Coxeter element.

Consider $\pi_n(X:Y:Z)=[X^n:Y^n:Z^n]$ the $n$-th Kummer cover of $\bp^2$ and denote by $\bar \cC_n$ the
preimage of $\bar \cC$ by $\pi_n$. The extended braid monodromy $T$ produces an extended braid monodromy 
$T_n\in \bb_{nk,1}^{nr+1}$ for $(\bar{\cC}_n,\pi_n^{-1} P_y,\bar{L}_\infty)$ (see \autoref{prop-kummer-bm}). 

\begin{prop}\label{prop-gen2}
Under the above conditions, a generic at infinity extended braid monodromy factorization is
obtained by adding to $T_n$ the braid $\Delta_{k+1}^{2 n}c_T^{-n}$, using the inclusion
$\bb_{k,1}\hookrightarrow\bb_{nk,1}$ (see diagram~\eqref{eq-diag-n}).
\end{prop}

\begin{proof}
It is a consequence of \autoref{prop-gen1} and \autoref{lema-image-central}.
\end{proof}

\subsection{From locally generic to generic.}\label{sec-lg-gen}
\mbox{}

As announced in \S\ref{sec-settings}, some types of non-generic braid monodromies 
can be easily deformed into generic braid monodromies. This is the case with locally
generic triples. Let $(\bar\cC,P_y,\bar{L}_\infty)$ be a locally generic triple 
(see \S\ref{sec-generic}) and $\cB$ the set of ramification values of the projection.
As already mentioned in \S\ref{sec-generic}, for each $t\in \cB$, there exists a finite
number of points $P_{t,1},\dots,P_{t,\kappa_t}\in L_t$ satisfying~\eqref{eq-cond-gen}.

\begin{prop}\label{prop-gen3}
Under the above conditions, let $\tau_t$ be the braid associated with $L_t$, $t\in \cB$. 
Then there is a factorization $\tau_t=\tau_{t,\kappa_t}\cdot\ldots\cdot\tau_{t,1}$, where 
$\tau_{t,1},\dots,\tau_{t,\kappa}$ are pairwise commuting braids and each one corresponds 
to one non-transversal point in~$L_t$. 

Moreover, replacing $\tau$ by $(\tau_{t,1},\dots,\tau_{t,\kappa_t})$ (the order does not matter) 
in the braid monodromy factorization of $(\bar\cC,P_y,\bar{L}_\infty)$ for each $t\in\cB$ 
produces a generic braid monodromy factorization of~$\bar{\cC}$.
\end{prop}

\begin{proof}
Since $P_y\notin \bar\cC$ the triple $(\bar\cC,P_y,\bar{L}_\infty)$ is fully horizontal, that is,
there are no vertical asymptotes. Consider $P_{t,1},\dots,P_{t,\kappa_t}\in L_t$. Due to the conic 
structure of $(\bc^2,\cC)_{P_{t,i}}$ one can find disjoint Milnor polydisks 
$\bb_{P_{t,i}}:=\bd_{t,x}\times \bd_{t,y}$ around each $P_{t,i}$ such that 
$\cC\cap \partial\bb_{P_{t,i}}\subset \partial \bd_{t,x}\times \bd_{t,y}$ (we need 
condition~\eqref{eq-cond-gen} to ensure this). 
Let us denote by $\mu_{t,i}$ the local monodromy around each $P_{t,i}$ based at $t_\varepsilon$ 
(close enough to $t$). Note that $\mu_{t,i}\in \bb_{m}$, where $m=1$ if and only if $P_{t,i}$ is a 
transversal intersection of $\cC$ and $L_t$. If $P_{t,i}$ is singular in $\cC$, then the link 
produced by closing $\mu_{t,i}$ is the link of the singularity at $P_{t,i}$. If $P_{t,i}$ is a 
single tangency, then $\mu_{t,i}$ is a half twist. Therefore, $\mu_{t,i}$ is trivial if and only if 
$P_{t,i}$ is a transversal intersection of $\cC$ and $L_t$. We disregard these trivial braids,
denote by $\kappa_t$ the number of non-trivial braids, and denote them by 
$\mu_{t,1},\dots,\mu_{t,\kappa_t}$ after reordering. 
Define by $\beta_t$ a path that joins $t_\varepsilon$ and $t$ consider $\tilde\beta_t$ the open braid
associated with $\beta_t$ and define $\tau_{t,i}:=\beta_t \cdot \mu_{t,i} \cdot \beta_t^{-1}$
(here ``$\cdot$'' just means juxtaposition). Since the strings inside each Milnor polydisk are different,
the braids $\tau_{t,i}$ automatically commute.

For the \emph{moreover} part, note that moving $P_y$ generically on $L_\infty$ causes just a small change 
of coordinates $X\mapsto X+\eta Y$ and hence each $L_t$ splits into $\kappa_t$ non-transversal vertical 
lines satisfying the genericity condition~\eqref{eq-cond-gen} at only one point. This shows the statement.
\end{proof}

\subsection{From generic at infinity to locally generic}\label{sec-gi-lg}
\mbox{}

In the more general case when triples $(\bar\cC,P_y,\bar{L}_\infty)$ are just generic at infinity, the 
change of projection might lead to more complicated situations, but in general the result is of the same
nature, that is, some braids produced by the monodromy around $\bar{L}_t$ in the non-generic braid monodromy 
need to be replaced by an appropriate factorization. We state some interesting particular cases.

\begin{prop}\label{prop-cusp}
Let us assume that $L_t$ intersects~$\bar{\cC}$ transversally at~$k-3$ points and it is tangent to an ordinary 
cusp. After conjugation the monodromy around $L_{t}$ produces the braid $\mu:=(\sigma_2\sigma_1)^2$. 

Moreover, $\mu$ can be replaced by the factorization $(\sigma_1^{\sigma_2},\sigma_2^3)$ so that the monodromy
becomes generic in a regular neighborhood of~$L_t$.
\end{prop}

\begin{figure}[ht]
\begin{center}
\includegraphics[scale=1]{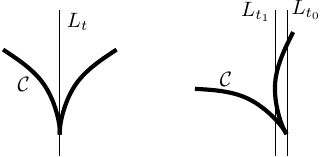}
\end{center}
\caption{Changing the projection.}
\label{fig:cusptangent}
\end{figure}

\begin{proof}
The local equation of $\cC$ at $P$ is $y^3-x^2=0$. Note that the local monodromy around $t=0$ is given
by the braid parametrized by $y^3=e^{4\pi \sqrt{-1}\lambda}$, $\lambda\in [0,1]$, that is, a rotation 
of angle $\frac{4}{3}\pi$ of the 3-rd roots of unity. This produces the braid $(\sigma_2\sigma_1)^2$.
A small perturbation of the projection point along 
$\bar{L}_\infty$ produces a change of variable $f_\eta:=y^3-(x-3\eta y)^2$. The discriminant of $f_\eta$
with respect to $y$ is $x^3(x - 4 \eta^3)$, which implies that the non-transversal 
vertical line $\bar{L}_0$ splits into $\bar L_{0}$ and $\bar{L}_{t_1}$, where $t_1=4\eta^3$.
It is straightforward to check that there is a cusp at $(0,0)$ (whose tangent $x=3\eta y$ is not vertical
if $\eta\neq 0$), a simple point at $(0,9\eta^2)$, a vertical tangency at $(t_1,4\eta^2)$, and a 
simple point at $(t_1,\eta^2)$ (see \autoref{fig:cusptangent}). 

Regardless of the value of $\eta$, there is a geometric basis on a local 
disk $\bd_{x,0}$ centered at $t=0$ and a choice of generators $\sigma_1$ and $\sigma_2$ in $\bb_3$ such that 
the local braid monodromy around $t=0$ is given by $\sigma_2^3$. Therefore, since the product should be
$(\sigma_2\sigma_1)^2$ and
$\sigma_2^3\cdot \sigma_1^{\sigma_2}=\sigma_2^2 \sigma_1 \sigma_2=(\sigma_2\sigma_1)^2$, then the local 
monodromy around $t=t_1$ should be given by $\sigma_1^{\sigma_2}$, which concludes the proof.
\end{proof}

Similar computations provide the following results for the simplest cases:

\begin{prop}
\label{prop-double}
Let us assume that $L_t$ intersects~$\bar{\cC}$ transversally at~$k-3$ points and it is tangent at one of the 
local branches of an ordinary double point. After conjugation the monodromy around $L_{t}$ produces the braid 
$\mu:=\sigma_1\sigma_2\sigma_1$. 

Moreover, $\mu$ can be replaced by the factorization $(\sigma_2^{\sigma_1},\sigma_1^2)$ so that the monodromy
becomes generic in a regular neighborhood of~$L_t$.
\end{prop}

\begin{prop}
\label{prop-tangency}
Let us assume that $L_t$ intersects~$\bar{\cC}$ transversally at~$k-m$ points and it is tangent to an inflection 
point of order $m$. After conjugation the monodromy around $L_{t}$ produces the braid 
$\mu:=\sigma_{m-1}\dots \sigma_{1}$. 

Moreover, $\mu$ can be replaced by the factorization $(\gs_{1},\dots,\sigma_{m-1})$ 
so that the monodromy becomes generic in a regular neighborhood of~$L_t$.
\end{prop}

\subsection{Line arrangements: from non-generic to generic.}
\mbox{}

We are also interested in the case where $\bar{\cC}$ is the union of a
fully horizontal curve $\bar{\cC}_0$ (such that $P_y\notin\bar{\cC}$) and some lines
$\bar{L}_1,\dots,\bar{L}_{\ell}$ passing through $P_y$ and different from $\bar{L}_\infty$.
In order to avoid technical problems, we restrict our attention to the 
case of line arrangements, which admits two approaches. One of them is to use
wiring diagrams~\cite{arvola} (which is an invariant equivalent to braid monodromy).
The other is more direct but we need some definitions.

\begin{dfn}
Let $n,k\in\bn$ and let $\ell\in\{0,1,\dots,k\}$. The $\ell$-\emph{shift} of $\bb_n$ into $\bb_{n+k}$
is the inclusion $\cub_\ell:\bb_n\to\bb_{n+k}$ such that $\cub(\sigma_j)=\sigma_{j+\ell}$,
$1\leq j<n$.
\end{dfn}

\begin{dfn}
The \emph{partial Garside element} of the strings $i,\dots,j$, $i\leq j$, is the image~$\Delta_{i,j}$
of the Garside element by the $(i-1)$-shift of $\bb_{j-i+1}$ into $\bb_n$.
\end{dfn}

Let $\bar{\cC}=\bar{L}_1\cup\dots\cup\bar{L}_n\cup\bar{L}_{n+1}\cup\dots\cup\bar{L}_{n+k}$ 
be a line arrangement such that $P_y$ is a point of multiplicity~$k$
and $P_y\in\bar{L}_{n+1}\cap\dots\cap\bar{L}_{n+k}$. Let $\bar{L}_\infty$ be a generic line through~$P_y$.
We want to construct a generic braid monodromy for $\bar{\cC}$ starting from 
the one of $(\bar \cC_h,P_y,\bar L_\infty)$, where $\bar \cC_h:=\bar{L}_1\cup\dots\cup\bar{L}_n$.
Recall that this braid monodromy is obtained as a representation
$\nabla:\pi_1(\bc\setminus\cB;t_0)\to\bb_n$ where $\cB$ is the set of $x$-coordinates
of the multiple points of $\cC_h$.
Let $\cB_0:=\cB\cup \{t_1,\dots,t_k\}$, where $L_{n+i}=\{x=t_i\}$; 
since the braids associated with the meridians around the points in $\cB_0\setminus\cB$ are trivial,
the above mapping defines a representation $\nabla_0:\pi_1(\bc\setminus\cB_0;t_0)\to\bb_n$ 
(which will be referred to as the \emph{augmented braid monodromy}).
A choice of an ordered geometric basis $(\gamma_1,\dots,\gamma_r)$ allows
to represent this braid monodromy by
$(\tau_1,\dots,\tau_r)\in\bb_n^r$; let $1\leq i_1<\dots<i_k\leq r$ be the indices of the braids corresponding to 
the vertical lines $L_{n+1},\dots,L_{n+k}$.

\begin{prop}\label{prop-def-lines}
Under the above conditions, let us decompose $\tau_i$ as 
$\beta_i*\alpha_i$
where 
\begin{equation}\label{eq-dec-ba}
\alpha_i=\prod_{s=1}^{m(i)} \Delta_{a_s(i),a_{s+1}(i)-1}^2,\quad
1=a_1(i)<\dots<a_{m(i)+1}(i)=n+1,
\end{equation}
is a 
product of squares of partial Garside elements of $\bb_n$. Then, a generic braid monodromy for~$\bar{\cC}$
is obtained by replacing the braids $\tau_i$ as follows:
\begin{enumerate}
\enet{\rm(\arabic{enumi})}
\item\label{prop-def-lines1} If $i_j<i<i_{j+1}$ (by convention $i_0=0$, $i_{k+1}=r+1$) then replace
$\tau_i$ by 
the sequence
$$
\left\{\beta_i*\Delta_{a_s(i),a_{s+1}(i)-1}^2\right\}_{m(i)\geq s\geq 1}
$$ 
and we take out the trivial braids (for $s$ such that $a_s(i)<a_{s+1}(i)-1$). 
\item\label{prop-def-lines2} If $i=i_j$ then replace $\tau_i$ by 
$$
\{(\beta_i\cdot\sigma_{n+j-1}^{-1}\cdot\ldots\cdot\sigma_{n+1}^{-1}\cdot\Delta_{a_{m(i)}(i),a_{m(i)+1}(i)}
\cdot\ldots\cdot\Delta_{a_{s+1}(i),a_{s+2}(i)})*\Delta_{a_s(i),a_{s+1}(i)}^2\}_{m(i)\geq s\geq 1}.
$$
\item\label{prop-def-lines3} Finally, add $\Delta_{n+1,n+k}^2$.
\end{enumerate}
\end{prop}

\begin{proof}
For simplicity, we can assume that $r=k$, since we may add vertical lines to the arrangement
and once the generic braid monodromy is obtained we may forget the strings corresponding to the added lines;
it is easily seen that this does not affect to the final result.

The augmented braid monodromy $\nabla_0$ described above may be computed as follows. 
Fix a generic line $\bar{L}_*$ through~$P_y$ close to $\bar L_\infty$ and fix a base point $P_*:=(t_0,y_*)\in L_*$
close to $P_y$. We fix an ordered geometric basis $(\mu_1,\dots,\mu_n)$
of $\pi_1(L_*\setminus\cC_h;P_*)$ such that $\mu_i$ is a meridian of $L_i$ and its pseudo-Coxeter element is the 
the boundary of a  disk $\bd_*$ in $L_*$ surrounding $L_*\cap\cC_h$
(which is also the negative boundary of a disk in $\bar{L}_*$ centered at $P_y$). 

In the affine plane $\bc^2=\bp^2\setminus\bar{L}_\infty$, we consider the horizontal line
$H_*$ passing through $P_*$. We can choose $P_*$ such that there is a 
disk $\bd_x$ containing $\cB_0\times\{y_*\}$ in its interior.
We fix an ordered geometric basis $(\tilde{\mu}_1,\dots,\tilde{\mu}_n,\tilde{\gamma}_1,\dots,\tilde{\gamma}_k)$
of $\pi_1(H_*\setminus\cC;P_*)$
such that:
\begin{itemize}
\item The meridians $\tilde{\mu_i}$ and $\mu_i$ are equal in $\pi_1(\bc^2\setminus\cC;P_*)$ (and we identify them from
now on). The product $\tilde{\mu}_n\cdot\ldots\cdot\tilde{\mu}_1$ is the boundary of a disk $\tilde{\bd}_*$ which is
isotopic to $\bd_*$ (in $\bc^2\setminus\cC$) and disjoint to $\bd_x$ (see \autoref{fig:hv}).
\item The meridians $(\tilde{\gamma}_1,\dots,\tilde{\gamma}_k)$ are obtained as follows. Pick a base point in
$\partial\bd_x$; all the meridians have as common part a path joining $P_*$ with this point avoiding
(counterclockwise) the disk $\tilde{\bd}_*$. These meridians project onto the meridians $\gamma_i$. 
\end{itemize}
It is then clear that 
$(\tilde{\gamma}_k\cdot\ldots\cdot\tilde{\gamma}_1)\cdot(\mu_d\cdot\ldots\cdot\mu_1)=1$ in 
$\pi_1(\bp^2\setminus\bar\cC;P_*)$.

Next step is to deform the projection point in $\bar L_\infty$ which induces a family of coordinate changes
$(x,y)\mapsto(x+s \varepsilon y,y)$ for a fixed $0<|\varepsilon|\ll 1$ and $s\in(0,1]$.  
For each new projection point
we obtain a braid monodromy $\nabla_s:\pi_1(\bc\setminus\cB_s;t_{0,s})\to\bb_{n+k}$, $s\in(0,1]$. 
We can fix small pairwise disjoint disks $\bd_1,\dots,\bd_k$ (respectively centered at $t_1,\dots,t_k$)
such that $\cB_s$ is contained in the interior of $\bigcup_{j=1}^k\bd_j$. Note that
$\#(\cB_s\cap\bd_j)=\#(\cC_h\cap L_{n+j})=m(j)$.
In order to compute the braid monodromy $\nabla_s$ we choose an ordered geometric basis 
$(\eta_{i,j})_{1\leq i\leq m(j)}^{1\leq j\leq k}$, such that $\prod_{i=1}^{m(j)}\eta_{i,j}=\gamma_j$.
The generic base fiber is the new vertical line passing through $P_*$. In this vertical line we can construct
disks isotopic to $\tilde{\bd}_*$ and $\bd_x$ containing respectively the intersections with
$L_1\cup\dots\cup L_n$ and $L_{n+1}\cup\dots\cup L_{n+k}$.

\begin{figure}[ht]
\centering
\subfigure[$\pi_1(H_*\setminus\bar{\cC};P_*)$]{
\includegraphics[scale=1]{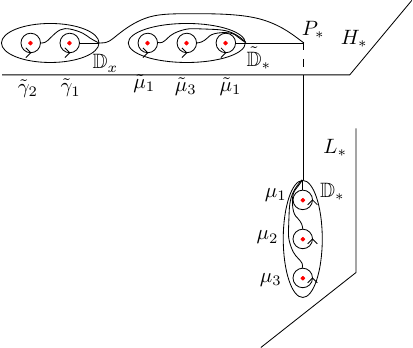}
\label{fig:hv}
}
\hfill
\subfigure[$\pi_1(L_*\setminus\bar{\cC};P_*)$]{
{\includegraphics[scale=1]{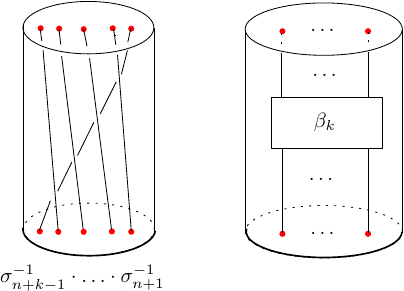}}
\label{fig:trenzabeta}
}
\label{fig:deformacion-rectas}
\caption{}
\end{figure}

We start with the computation of $\nabla_1(\gamma_j)$. The decomposition
$\tau_j=\beta_j*\alpha_j$ has a geometric meaning: $\beta_j$ is the braid starting from the
vertical line through $P_*$ and ending in a vertical line close to $L_{n+j}$ while
$\alpha_j$ is the braid obtained by turning around $L_{n+j}$ as \eqref{eq-dec-ba} justifies.
Let us produce a similar decomposition for $\nabla_1(\gamma_j)=\tilde{\beta}_j*\tilde{\alpha}_j$ in $\bb_{n+k}$.
The motion producing $\beta_j$ induces now a braid 
$\tilde{\beta}_j$ as the product of two commuting braids:
$\beta_j$ and $\sigma_{n+j-1}^{-1}\cdot\ldots\cdot\sigma_{n+1}^{-1}$,
see \autoref{fig:trenzabeta}.
The first one corresponds to the motion of the points in the disk $\tilde{\bd}_*$
and the second one is obtained by considering the motion of the $i^{\rm th}$ point to the
boundary (behind the other ones), and then how this point approaches to $\tilde{\bd}_*$.

\begin{figure}[ht]
\begin{center}
\includegraphics[scale=1]{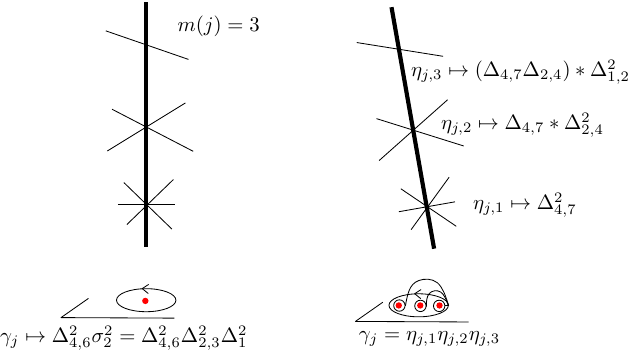}
\end{center}
\caption{Braid $\tilde{\alpha}_j$.}
\label{fig:rectasrotar}
\end{figure}

The braid $\tilde{\alpha_j}$ is decomposed as a product of conjugate of partial Garside elements
leading to the images of $\eta_{i,j}$ as in \ref{prop-def-lines2}, see \autoref{fig:rectasrotar}.

For the case $k<r$,
the decomposition in~\ref{prop-def-lines1} is obtained
is the same one by forgetting the strings corresponding to the verticals line not in~$\cC$.

The final step consists in the choice
of a new line at infinity close to $\bar L_\infty$ (without changing the projection point);
the last lines are no more parallel since
the multiple point~$P_y$ becomes affine and we must add the braid of~\ref{prop-def-lines3}.
\end{proof}

\begin{ejm}
Let us consider the line arrangement of \autoref{fig:ejemplo_recta}.
In order to obtain the augmented braid monodromy $\nabla_0$ it is necessary to add an extra
vertical line $L$ passing through the double point $L_1\cap L_3$ of $\cC$. It is easy to
see that $\nabla_0$ is given by $(\sigma_2^2,1,\sigma_2*\sigma_1^2,\sigma_1^2)$.
\begin{figure}[ht]
\begin{center}
\includegraphics[scale=1]{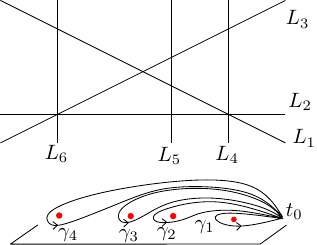}
\end{center}
\caption{From non-generic to generic braid monodromy.}
\label{fig:ejemplo_recta}
\end{figure}
The generic braid monodromy is given by:
\begin{equation*}
(\Delta_{2,4}^2,\Delta_{2,4}*\sigma_1^2,
(\sigma_3^2)^{\sigma_4},(\sigma_4^{-1}\sigma_3)*\sigma_2^2,(\sigma_4^{-1}\sigma_3\sigma_2)*\sigma_1^2,
\sigma_2*\sigma_1^2,
(\sigma_3^2)^{\sigma_4\sigma_5},(\sigma_5^{-1}\sigma_4^{-1}\sigma_3)*\Delta_{1,3}^2).
\end{equation*}
\end{ejm}

\section{Transformations of curves by Kummer covers.}\label{sec-trans}

Let $\bar{\cC}$ be a (reduced) projective curve of degree~$k$ of equation $F_k(x,y,z)=0$
and let $\bar{\cC}_n$ be its transform by a Kummer cover $\pi_n$, $n>1$. Note 
that $\bar{\cC}_n$ is a projective curve of degree~$n k$ of equation $F_k(x^n,y^n,z^n)=0$.
Another obvious remark is that if $\bar{\cC}$ is reducible so is $\bar{\cC}_n$. The converse
is not true as we will see in~\S\ref{sec-ex}.

We will briefly analyze the singularities of $\bar{\cC}_n$ in terms of $\bar{\cC}$. 
For convenience we distinguish three types of points in $\bp^2$.

\begin{dfn}
Let $P\in\bp^2$ such that $P:=[x_0:y_0:z_0]$. We say
that $P$ is a point of \emph{type $(\bc^*)^2$ } (or simply of \emph{type $2$}) if 
$x_0 y_0 z_0\neq 0$. If $x_0=0$ but $y_0 z_0\neq 0$ the point is said to be of 
\emph{type~$\bc^*_x$} (types $\bc^*_y$ and $\bc^*_z$ are defined accordingly).
Such points will also be referred to as \emph{type $1$} points. The corresponding line
(either $L_X:=\{X=0\}$, $L_Y:=\{Y=0\}$, or $L_Z:=\{Z=0\}$) a type-$1$ lies on will be referred to as 
their \emph{axis}. The remaining points $P_x:=[1:0:0]$, $P_y:=[0:1:0]$, and 
$P_z:=[0:0:1]$ will be called \emph{vertices} (or type~0 points) and their axes are
the two lines (either $L_X$, $L_Y$, or $L_Z$) they lie on.
\end{dfn}

\begin{obs}
Note that a point of type $\ell$, $\ell=0,1,2$ in $\bp^2$ has exactly $n^\ell$ 
preimages under~$\pi_n$. It is also clear that the local type of $\cC_n$ at any
two points on the same fiber are analytically equivalent.
\end{obs}

\begin{lema}\label{lema-local}
Let $P\in\bp^2$ be a point of type~$\ell$ and $Q\in \pi_n^{-1}(P)$.There exist
local coordinates $(u_0,v_0)$ and $(u_1,v_1)$ centered at $Q$ and $P$, respectively,
such that:
\begin{enumerate}
\enet{\rm(\arabic{enumi})}
 \item If $\ell=2$, then $(u_1,v_1)=\pi_n(u_0,v_0)=(u_0,v_0)$. 
 \item If $\ell=1$, then $(u_1,v_1)=\pi_n(u_0,v_0)=(u_0^n,v_0)$, where $u_0=0$ and $u_1=0$
are the local equations (at $Q$ and $P$, respectively) of the axes containing the points.
 \item If $\ell=0$, then $(u_1,v_1)=\pi_n(u_0^n,v_0)=(u_0^n,v_0^n)$, where $u_i=0$ and $v_i=0$
are the local equations of the axes containing $P,Q$.
\end{enumerate}
\end{lema}

\begin{proof}
For the points of type~$2$, it is a straightforward consequence of fact that the cover $\pi_n$
is unramified at~$P$. If $P$ is of type~$1$, let us assume that $Q:=[0:t:1]$, $t\in\bc^*$, and then
$P=[0:t^n:1]$. We choose $(u_0,v_0)=(x,y-t)$ and $(u_1,v_1)=(x,\sqrt[n]{t^n+y})$. If $P$ is of type~$0$, we may assume $Q=[0:0:1]$; we keep $x,y$ as coordinates for $P$ and $Q$.
\end{proof}

\begin{prop}\label{prop-sing}
Let $P\in\bp^2$ be a point of type~$\ell$ and $Q\in \pi_n^{-1}(P)$. One has the following:
\begin{enumerate}
\enet{\rm(\arabic{enumi})}
 \item If $\ell=2$, then $(\bar{\cC},P)$ and $(\bar{\cC}_n,Q)$ are analytically isomorphic. 
 \item If $\ell=1$, then $(\bar{\cC}_n,Q)$ is a singular point of type~$1$ if and only if~$m>1$,
where $m:=(\bar{\cC}\cdot\bar{L})_P$ and $\bar L$ is the axis of $P$.
 \item If $\ell=0$, then $(\bar{\cC}_n,Q)$ is a singular point.
\end{enumerate}
\end{prop}

\begin{proof}
It is a straightforward consequence of \autoref{lema-local}. Assume 
$P=Q=(0,0)$ are points of type~$\ell$ on $\bar\cC$ and $f(u_1,v_1)$ is a local equation of $(\bar\cC,P)$. 
If $\ell=2$, then $f(u_0,v_0)$ is a local equation of $(\bar\cC_n,Q)$.
If $\ell=1$  
then $f(u_0^n,v_0)$ is a local equation of 
$(\bar\cC_n,Q)$. Finally, if $\ell=0$,
then $f(u_0^n,v_0^n)$ is a local equation of $(\bar\cC_n,Q)$.
\end{proof}

\begin{ejm}\label{ex-singpin}
In some cases, we can be more explicit about the singularity type of $(\bar\cC_n,Q)$.
\begin{enumerate}
\item If $P$ is of type~$1$, $(\bar{\cC},P)$ is smooth and $m:=(\bar{\cC}\cdot\bar{L})_P$
then $(\bar{\cC}_n,Q)$ has the same topological type as $u_0^n-v_0^m=0$. In particular, if $n=2$,
then $(\bar{\cC}_n,Q)$ is of type $\ba_{m-1}$.
\item If $P$ is of type~$0$, $(\bar{\cC},P)$ is smooth and transverse to the axes
then $(\bar{\cC}_n,Q)$ is an ordinary multiple point of multiplicity~$n$.
\end{enumerate}

\end{ejm}

In order to better describe singular points of type $0$ and $1$ of $\bar \cC_n$ we will introduce some notation.
Let $P\in\bp^2$ be a point of type~$\ell=0,1$ and $Q\in \pi_n^{-1}(P)$ a singular point of 
$\bar \cC_n$. Denote by $\mu_P$ (resp. $\mu_Q$) the Milnor number of $\bar \cC$ at $P$ (resp. $\bar \cC_n$ at $Q$),
also, denote by $\delta_1,\dots,\delta_r$ the local branches of $\bar \cC$ at $P$ and consider
$\tilde\delta_i:=\pi_n^{-1}(\delta_i)$. Define $\mu_{P,\delta_i}$ (resp. $\mu_{Q,\tilde\delta_i}$) as the 
Milnor number of the singularity $(\delta_i,P)$ (resp. $(\tilde\delta_i,Q)$). 
Since $\ell=0,1$, then $P$ and $Q$ belong to either exactly one or two axes. 
If $P$ and $Q$ belong to an axis $\bar L$, then $m_P^{\bar L}:=(\bar \cC\cdot \bar L)_P$ and 
$m_{P,\delta_i}^{\bar L}:=(\delta_i\cdot \bar L)_P$ (analogous notation for $Q$ and $\tilde\delta_i$).
More specific details about singular points of types~$0$ and $1$ can be described as follows:

\begin{prop}\label{prop-sing-kummer}
Under the above conditions and notation one has the following:
\begin{enumerate}
\enet{\rm(\arabic{enumi})}
 \item\label{prop-sing-kummer1} 
For $\ell=1$, $P$ belongs to a unique axis $\bar L$ and
\begin{enumerate}
\eneti{\rm(\alph{enumii})}
\item $\mu_Q=n \mu_P+(m_P^{\bar L}-1)(n-1)$,
\item $\mu_{Q,\tilde\delta_i}=n \mu_{P,\delta_i}+(m_{P,\delta_i}^{\bar L}-1)(n-1)$,
\item If $i\neq j$ then $(\tilde\delta_i\cdot \tilde\delta_j)_Q=n (\delta_i\cdot \delta_j)_P$,
\item If $r_i:=\gcd(n,m_{P,\delta_i}^{\bar L})$, then  $\tilde\delta_i$ has $r_i$ irreducible components.
\end{enumerate}

 \item\label{prop-sing-kummer0}
For $\ell=0$, $P$ belongs to exactly two axes $\bar L_1$ and $\bar L_2$ 
\begin{enumerate}
\eneti{\rm(\alph{enumii})}
\item $\mu_Q=n^2 \mu_P+(n-1)(n(m_P^{\bar L_1}+m_P^{\bar L_2})-1)$,
\item $\mu_{Q,\tilde\delta_i}=n^2 \mu_{P,\delta_i}+(n-1)(n(m_{P,\delta_i}^{\bar L_1}+m_{P,\delta_i}^{\bar L_2}-1)-1)$,
$m_{Q,\tilde\delta_i}^{\bar L_j}=n m_{P,\delta_i}^{\bar L_j}$,
\item If $i\neq j$ then $(\tilde\delta_i\cdot \tilde\delta_j)_Q=n^2 (\delta_i\cdot \delta_j)_P$,
\item If $r_i:=\gcd(n,m_{P,\delta_i}^{\bar L_1},m_{P,\delta_i}^{\bar L_2})$,
then $\tilde\delta_i$ has $n r_i$ irreducible components
which are analytically isomorphic to each other for any fixed~$i$. 
\end{enumerate}
\end{enumerate}
\end{prop}

\begin{proof}
For part~\eqref{prop-sing-kummer1} note that
$$
\array{c}
\mu_Q=\dim_{\bc} \frac{\bc\{x,y\}}{(x^{n-1}f_x(x^n,y),f_y(x^n,y))} = 
\dim_{\bc} \frac{\bc\{x,y\}}{(x^{n-1},f_y(x,y))} + \dim_{\bc} \frac{\bc\{x,y\}}{(f_x(x^n,y),f_y(x^n,y))} = \\ \\
(n-1) \dim_{\bc} \frac{\bc\{x,y\}}{(x,f_y(x,y))} + n \dim_{\bc} \frac{\bc\{x^n,y\}}{(f_x(x^n,y),f_y(x^n,y))} = 
(n-1) (m_P^{\bar L}-1) + n \mu_P.
\endarray
$$
The same proof applies to $\mu_{Q,\tilde\delta_i}$. Also
$$(\tilde\delta_i\cdot \tilde\delta_j)_Q= 
\dim_{\bc} \frac{\bc\{x,y\}}{(f_i(x^n,y),f_j(x^n,y))} = 
n \dim_{\bc} \frac{\bc\{x^n,y\}}{(f_i(x^n,y),f_j(x^n,y))} = 
n (\delta_i\cdot \delta_j)_P,$$
where $f_i$ (resp. $f_j$) is a local equation for $\delta_i$ (resp. $\delta_j$).

Finally, we can describe the irreducible branch $\delta_i$ as a Puiseux factorization of type 
$f_i(x,y)=\prod_{j=1}^{\nu} (y-s(\xi_\nu^j x^{1/\nu}))$, where $\nu$ is a multiple of 
$m=m_{P,\delta_i}^{\bar L}$, the order of $f_i$ in $y$. 
Note that $f_i(x^n,y)=\prod_{j=1}^{\nu} (y-s(\xi_\nu^j x^{n/\nu}))$.
Assume $r$ is a common divisor, that is, $n=rn'$, $m=rm'$ (and hence $\nu=r\nu'$) then 
$$\prod_{j=1}^{\nu'} (y-s(\xi_\nu^{rj} x^{rn'/\nu}))=\prod_{j=1}^{\nu'} (y-s(\xi_{\nu'}^{j} x^{n'/\nu'}))$$
is invariant under the Galois conjugation by $\nu'$-roots of unity and hence it is a convergent series in~$\bc\{x,y\}$,
that is, a union of branches. Therefore the result follows.

The proof of part~\ref{prop-sing-kummer0} is analogous.
\end{proof}

\section{Examples and applications.}\label{sec-ex}

In this section the results obtained in \S\ref{sec-pencils}-\ref{sec-cover-braids} 
will be applied to produce generic braid monodromies of curves.
We will follow the conventions introduced in~\cite{acc:01a}. In the language
of the previous section, diagram systems on $\bc$ for the starting curve $\cC$ will 
be chosen by joining the points with segments in a decreasing order according to the 
lexicographic order.

\subsection{Smooth curves}\label{subsec-smooth}
\mbox{}

We start with the braid monodromy of a smooth curve considered as a Fermat curve 
$\bar \cC_n=\{x^n-y^n+z^n=0\}\subset \bp^2$, via the $n$-th Kummer cover of the line~$\bar{\cC}$ 
of equation $x-y+z=0$.
Using the conventions of \S\ref{sec-pencils}, one can obtain an extended braid monodromy
$\tilde{\nabla}:\pi_1(\bc\setminus\tilde{B};t_0)\to\bb_{1,1}$, where
$\gamma_1,\gamma_2$ is a geometric basis of the free group $\pi_1(\bc\setminus\tilde{B};t_0)$
and $\tilde{\nabla}(\gamma_1)=1$, $\tilde{\nabla}(\gamma_2)=\sigma_1^2$ as shown in~\autoref{fig:caso1}.
\begin{figure}[ht]
\begin{center}
\includegraphics[scale=1]{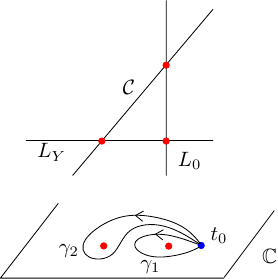}
\end{center}
\caption{A generic line.}
\label{fig:caso1}
\end{figure}
Using \autoref{ex-0} and the commutative diagram~\eqref{eq-diag-n}, one can obtain a braid monodromy
$\nabla_n:\bff_{n}\to\bb_{n}$ for $\cC_n$ as follows:
\begin{equation}
\nabla_n(\tilde{\gamma}_j)=\sigma_{n-1}\cdot\ldots\cdot\sigma_1,
\end{equation}
where the basis $\tilde{\gamma}_j$, $j=1,\dots,n$ is obtained as in 
\autoref{lema-geometric-basis-cover}, by forgetting the last term since $\bar{L}_0$ is transversal.

However, note that $\nabla_n$ is not a generic braid monodromy since the projection from $P_y$
contains $n$ special fibers $\bar L_{\xi_n^i}$, $i=0,...,n-1$ with an order $n$ tangency at
$[\xi_n^i:0:-\xi_n^i]$. Still, $\nabla_n$ is useful to compute the fundamental group of the complement
of a smooth curve (using \autoref{thm-zvk3}). This group is generated by $\mu_1,\dots,\mu_d$
with a relation $\mu_d\cdot\ldots\cdot\mu_1$ and since
$$
\mu_i^{\sigma_{n-1}\cdot\ldots\cdot\sigma_1}=\mu_{i+1},\quad 1\leq i<d
$$
the group is cyclic of order~$d$, as obtained by Zariski in~\cite{zr:29}.

By the previous discussion, in order to obtain a generic braid monodromy, it is enough to use
\autoref{prop-tangency} and obtain the following.

\begin{prop}[{\cite[Theorem 1, p.~120]{mz:81}}]
Let $\bar{\cC}$ be a smooth curve of degree~$n$. Then $\bar{\cC}$ 
induces a braid monodromy factorization of $\Delta_n^2$ given by
$$
(\overbrace{\sigma_{1},\dots,\sigma_{n-1}}^{1^{\text{st}}\text{ package}},\dots,
\overbrace{\sigma_{1},\dots,\sigma_{n-1}}^{n^{\text{th}}\text{ package}}).
$$
\end{prop}

\subsection{Zariski sextics}\label{subsec-zs}
\mbox{}

In~\cite{zr:29} Zariski~showed that the fundamental group of a sextic with
six cusps on a conic is~$\bz/2*\bz/3$ and the family of such sextics is irreducible,
so they all have equivalent braid monodromies. We are going to compute one in two
steps: first a (fully horizontal) non-generic braid monodromy of such a sextic will be 
obtained from a simpler curve via a Kummer covering and second, a deformation will be 
performed to compute the desired generic braid monodromy.

Consider a conic as in \autoref{fig:caso2}. It is tangent to two of the axes
and transversal to the third one, e.g., the conic~$\bar{\cC}=\{x^2+y^2-2 x z-2 y z+z^2=0\}$.
Using the conventions of \S\ref{sec-pencils}, the projection from $P_y$ has three non-generic
fibers at $\tilde{B}=\{0,1,2\}$. Consider $t_0=3$ and $\gamma_1,\gamma_2, \gamma_3$ meridians 
around $2$, $1$, and $0$ resp. (see \autoref{fig:caso2}) forming an ordered geometric basis
of the free group $\pi_1(\bc\setminus\tilde{B};t_0)$. One can obtain an extended braid monodromy 
$\tilde{\nabla}:\pi_1(\bc\setminus\tilde{B};t_0)\to\bb_{2,1}$ as follows 
$$
\tilde{\nabla}(\gamma_1)=\sigma_1,\quad
\tilde{\nabla}(\gamma_2)=\sigma_1 *\sigma_2^4,\quad
\tilde{\nabla}(\gamma_3)=(\sigma_1\sigma_2^2)*\sigma_1.
$$
\begin{figure}[ht]
\begin{center}
\includegraphics[scale=1]{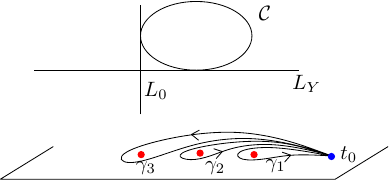}
\end{center}
\caption{A conic tangent to two axes.}
\label{fig:caso2}
\end{figure}
After conjugating by the braid~$(\sigma_1\sigma_2^2)$, the above representation becomes:
$$
\tilde{\nabla}(\gamma_1)=\sigma_1^{\sigma_2^2},\quad
\tilde{\nabla}(\gamma_2)=\sigma_2^4,\quad
\tilde{\nabla}(\gamma_3)=\sigma_1.
$$

Let us  consider the $3$-rd Kummer cover of $\bar{\cC}$. 
The curve~$\bar{\cC}_3$ is a curve of degree~$6$. Since $\bar{\cC}$ is smooth
and intersects transversally $\bar{L}_\infty$, the singular points
are contained in~$\bar{L}_0\cup\bar{L}_Y$ (\autoref{prop-sing}). 
Using \autoref{prop-sing-kummer}, the curve $\bar{\cC}_3$~possesses six ordinary cusps 
(which are in a conic, namely the union of the lines $\pi_3^{-1}(\bar L_0)=\bar L_0$ and
$\pi_3^{-1}(\bar L_Y)=\bar L_Y$).

Using a radial system of generators as in \S\ref{sec-radial}, the map
$\hat{\cub}_3:\bb_{2,1}\to\bb_{6}$ is given (see \autoref{prop-motion2}) by:
\begin{alignat*}{5}
\sigma_1&\mapsto\radu_1\radu_3\radu_5\\
\sigma_{2}^{2}&\mapsto\radu_5\radu_4\radu_3\radu_2\radu_3^{-1}\radu_5^{-1}=
(\radu_5 \radu_4 \radu_2^{-1})*(\radu_4 \radu_3),&&
\end{alignat*}
where the generators $\radu_i$ on the right-hand side are defined as
$\radu_1:=\rad_{1,1},\radu_2:=\rad_{2,1},\dots,\radu_5:=\rad_{1,3}$ 
(with the right-lexicographic order) for simplicity. 
Let us conjugate the result by $\radu_5\radu_3$:
\begin{equation}
\label{eq-rho3}
\array{rcl}
\sigma_1&\mapsto&\radu_1\radu_3\radu_5\\
\sigma_{2}^{2}&\mapsto&\radu_3^{-1}\radu_4\radu_3\radu_2=\radu_4*(\radu_3 \radu_2).
\endarray
\end{equation}

\begin{prop}
There is a geometric basis, as in Lemma{\rm~\ref{lema-geometric-basis-cover}},
such that $\nabla_3:\bff_{7}\to\bb_{6}$ is given as follows:
$$
\begin{matrix}
\tilde\gamma_1&=&\gamma_1&\mapsto&(\radu_1\radu_3\radu_5)^{\radu_4*(\radu_3\radu_2)}=
&(\radu_1\radu_3\radu_5)^{\radu_4\radu_3\radu_2}\\
\tilde\gamma_2&=&\gamma_{2}&\mapsto&&\radu_4*(\radu_3 \radu_2)^2\\
\tilde\gamma_3&=&\gamma_1^{\gamma_3}&\mapsto&(\radu_1\radu_3\radu_5)^{\radu_4*(\radu_3\radu_2)\radu_1\radu_3\radu_5}
=&(\radu_1\radu_3\radu_5)^{\radu_4\radu_3\radu_2\radu_1\radu_3\radu_5}\\
\tilde\gamma_4&=&\gamma_2^{\gamma_3}&\mapsto&&((\radu_3 \radu_2)^2)^{\radu_4^{-1}\radu_1\radu_3\radu_5}\\
\tilde\gamma_5&=&\gamma_1^{\gamma_3^2}&\mapsto&(\radu_1\radu_3\radu_5)^{\radu_4*(\radu_3\radu_2)\radu_1^2\radu_3^2\radu_5^2}
=&(\radu_1\radu_3\radu_5)^{\radu_4\radu_3\radu_2\radu_1^2\radu_3^2\radu_5^2}\\
\tilde\gamma_6&=&\gamma_2^{\gamma_3^2}&\mapsto&&((\radu_3 \radu_2)^2)^{\radu_4^{-1}\radu_1^2\radu_3^2\radu_5^2}\\
\tilde\gamma_7&=&\gamma_3^3&\mapsto&&\radu_1^3\radu_3^3\radu_5^3
\end{matrix}
$$
\end{prop}

As mentioned above, this braid monodromy is not generic. However, the triple $(\bar \cC_3,P_y,L_\infty)$ 
is fully horizontal and transversal at infinity. One can deform $\nabla_3$ as described in \S\ref{sec-deformation}. 
First, a deformation is performed to obtain as in \S\ref{sec-gi-lg} to produce a locally generic triple. 
Then, a move as described in~\S\ref{sec-lg-gen} will produce a generic braid monodromy. 

Also note that the non-generic fibers of the projection from $P_y$ are 
$$\tilde B_3=\{t_1=\sqrt[3]{2},t_2=1,t_3=\xi_3\sqrt[3]{2},t_4=\xi_3,t_5=\xi_3^2\sqrt[3]{2},t_6=\xi_3^2,t_7=0\},$$ 
where $\tilde \gamma_i$ is a meridian around $t_i$ and:
\begin{enumerate}
 \item\label{enum-lg1} 
$\bar L_{t_i}$, $i=1,3,5$ correspond to vertical lines containing three simple tangencies (hence locally
generic fibers).
 \item\label{enum-nlg} 
$\bar L_{t_i}$, $i=2,4,6$ correspond to vertical lines passing tangent to an ordinary cusp and transversal
to three smooth points of~$\bar \cC_3$.
 \item\label{enum-lg2} 
$\bar L_0$ corresponds to a vertical line passing through three ordinary cusps (hence a locally generic fiber).
\end{enumerate}

\begin{paso}
This step is necessary at the fibers of type~\eqref{enum-nlg} above. According to \autoref{prop-cusp}
the braid obtained as image of $\tilde \gamma_2$, that is, 
$\radu_4*(\radu_3 \radu_2)^2=((\radu_4*\radu_3)\cdot (\radu_2))^2$ has to be replaced by 
\begin{equation}
\label{eq-1-2}
(\radu_2^{\radu_4*\radu_3},\radu_4*\radu_3^3)=(\radu_2^{\radu_3\radu_4^{-1}},(\radu_3^3)^{\radu_4^{-1}}).
\end{equation}
Analogously can be done with $\nabla_3(\tilde\gamma_4)$ and $\nabla_3(\tilde\gamma_6)$. 
\end{paso}

\begin{paso}
This is necessary at the locally generic fibers described above. 
Applying \autoref{prop-gen3} to the fibers of types~\eqref{enum-lg1} and \eqref{enum-lg2} one obtains that,
for instance, $\nabla_3(\tilde\gamma_1)=(\radu_1\radu_3\radu_5)^{\radu_4\radu_3\radu_2}$ is replaced by
\begin{equation}
\label{eq-3-4}
(\radu_1^{\radu_4\radu_3\radu_2},\radu_3^{\radu_4\radu_3\radu_2},\radu_5^{\radu_4\radu_3\radu_2})=
(\radu_1^{\radu_2},\radu_4,\radu_5^{\radu_4\radu_3\radu_2}).
\end{equation}
Analogously can be done with $\nabla_3(\tilde\gamma_3)$, $\nabla_3(\tilde\gamma_5)$, and $\nabla_3(\tilde\gamma_7)$. 
\end{paso}

However, instead of working out each $\nabla_3(\tilde\gamma_i)$ separately we can note that 
$\nabla_3(\tilde\gamma_{i+2j})=\nabla_3(\tilde\gamma_i)^{\radu_1^j\radu_3^j\radu_5^j}$, ($i=1,2$, $j=0,1,2$).
Therefore one can also work out the first set of braids coming from $\nabla_3(\tilde\gamma_1)$ and 
$\nabla_3(\tilde\gamma_2)$ and conjugate to obtain the rest of the braids coming from $\nabla_3(\tilde\gamma_i)$
$i=3,...,6$. Finally attach the braids $(\radu_1^3,\radu_3^3,\radu_5^3)$ obtained from $\nabla_3(\tilde\gamma_7)$. 
From~\eqref{eq-1-2} and~\eqref{eq-3-4} one obtains
$$
(\radu_1^{\radu_2},\radu_4,\radu_5^{\radu_4\radu_3\radu_2},\radu_2^{\radu_3\radu_4^{-1}},(\radu_3^3)^{\radu_4^{-1}}).
$$
Using the sequence of Hurwitz moves $\hur_2^{-1}\hur_3^{-1}\hur_4^{-1}\hur_2\hur_3$ (see~\eqref{eq-hurwitz}) 
one obtains the first block:

\begin{equation}\label{eq-z6-1}
(\sigma_1^{\sigma_2},\sigma_2^{\sigma_3},\sigma_3^3,\sigma_4^{\sigma_5^{-1}\sigma_3^{-1}},\sigma_4).
\end{equation}
After conjugating \eqref{eq-z6-1} by $\sigma_1\sigma_3\sigma_5$ and applying $\hur_2$ one obtains
the second block:
\begin{equation}\label{eq-z6-2}
(\sigma_2^{\sigma_3},\sigma_3^3,\sigma_2^{\sigma_3^{-1}\sigma_1},\sigma_4,\sigma_4^{\sigma_3\sigma_5}). 
\end{equation}
The third block is obtained from \eqref{eq-z6-2} using conjugation by $\sigma_1\sigma_3\sigma_5$
and applying $\hur_1\hur_4^{-1}$:
\begin{equation}\label{eq-z6-3}
(\sigma_3^3,\sigma_2^{\sigma_3^{-1}\sigma_1},\sigma_2^{\sigma_1^2},\sigma_4,\sigma_4^{\sigma_3\sigma_5}). 
\end{equation}

\begin{prop}
A generic braid monodromy for $\bar{\cC}_3$ is given by:
$$
(\sigma_1^{\sigma_2},\sigma_2^{\sigma_3},\sigma_3^3,\sigma_4^{\sigma_5^{-1}\sigma_3^{-1}},\sigma_4,
\sigma_2^{\sigma_3},\sigma_3^3,\sigma_2^{\sigma_3^{-1}\sigma_1},\sigma_4,\sigma_4^{\sigma_3\sigma_5},
\sigma_3^3,\sigma_2^{\sigma_3^{-1}\sigma_1},\sigma_1^3,\sigma_1^{\sigma_2},\sigma_4,\sigma_4^{\sigma_3\sigma_5},
\sigma_3^3,\sigma_5^3).
$$
\end{prop}

\begin{obs}
This result is also obtained in~\cite[Theorem 1(3), p.~160]{mz:81}.
It is straightforward to compute the fundamental group of the complement of $\bar{\cC}_3$ and to retrieve
Zariski's computation. Moreover, this braid monodromy allows us to compute the homotopy type
of $\bc^2\setminus\cC_3$ (for a generic choice of line at infinity) using the method in~\cite{li:86}. 
It is easy to see that
$$
\bc^2\setminus \cC_3 \simeq \left( \bs^3 \setminus \{\text{trefoil knot}\} \right) \vee \bigvee_{i=1}^{13} \bs^2.
$$
\end{obs}

\subsection{Dual of a smooth cubic}\label{subsec-c6-9c}
\mbox{}

The dual of a smooth cubic is a sextic with~$9$ cusps. Kummer covers allow to recover one
of these curves easily.

Consider a conic $\bar\cC:=\{x^2+y^2+z^2-2 (x y+x z+y z)=0\}$ as in \autoref{fig:caso3}. 
Projecting from $P_y=[0:1:0]$ as usual, one obtains two non-generic fibers $\bar L_\infty$
and $\bar L_0$ with tangencies $P_3:=[1:1:0]$ and $P_1:=[0:1:1]$ respectively. Also note that 
$\cC$ is tangent to $\bar L_Y$ at $P_2:=[1:0:1]$. Using the conventions of \S\ref{sec-pencils}, 
an extended braid monodromy $\tilde{\nabla}:\pi_1(\bc\setminus\tilde{B};t_0)\to\bb_{2,1}$, where 
$\tilde{B}=\{P_1,P_2,P_3\}$, $\gamma_1,\gamma_2$ are meridians around $P_2$ and $P_1$ respectively 
forming an ordered basis of the free group $\pi_1(\bc\setminus\tilde{B};t_0)$ 
and $\tilde{\nabla}(\gamma_1)=\sigma_2^4$, $\tilde{\nabla}(\gamma_2)=\sigma_2^2*\sigma_1$. 
For simplicity, we conjugate this monodromy by~$\sigma_2^2$ and obtain
\begin{equation}
\label{eq-bm-conic}
\gamma_1\mapsto\sigma_2^4,\qquad
\gamma_2\mapsto\sigma_1.
\end{equation}

\begin{figure}[ht]
\begin{center}
\includegraphics[scale=1]{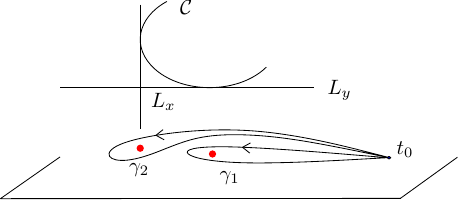}
\end{center}
\caption{Tritangent conic.}
\label{fig:caso3}
\end{figure}

The preimage of $\bar\cC$ by the Kummer cover of order $3$ is a sextic~$\bar{\cC}_3$ with nine cusps
as it can be deduced from \autoref{ex-singpin}. 
Let us compute the braid monodromy $(\bar\cC_3,P_y,\bar L_\infty)$. Following the ideas in~\S\ref{subsec-zs} 
one immediately obtains 
$$
\begin{matrix}
\gamma_1&\mapsto&((\radu_3 \radu_2)^2)^{\radu_4^{-1}}\\
\gamma_1^{\gamma_3}&\mapsto&((\radu_3 \radu_2)^2)^{\radu_4^{-1}\radu_1\radu_3\radu_5}\\
\gamma_1^{\gamma_3^2}&\mapsto&((\radu_3 \radu_2)^2)^{\radu_4^{-1}\radu_1^2\radu_3^2\radu_5^2}\\
\gamma_3^3&\mapsto&\radu_1^3\radu_3^3\radu_5^3.
\end{matrix}
$$
Since the line $\bar L_\infty$ is not generic, one can apply \autoref{prop-gen2} to obtain a braid monodromy
that is generic at infinity. The pseudo-Coxeter element for the monodromy~\eqref{eq-bm-conic} is 
$c:=\sigma_1\sigma_2^4$ and hence $\Delta_3^6c^{-3}=(\Delta_3^2c^{-1})^3=(\sigma_1^3)^{\sigma_2^2}\in \bb_{2,1}$,
whose image in $\bb_{6}$ via $\hat\rho_3$ (see~\eqref{eq-rho3}) is 
$(\radu_1^3\radu_3^3\radu_5^3)^{\radu_4\radu_3\radu_2\radu_4^{-1}}$. Therefore
$$
(((\sigma_3 \sigma_2)^2)^{\sigma_4^{-1}},((\sigma_3 \sigma_2)^2)^{\sigma_4^{-1}\sigma_1\sigma_3\sigma_5},
((\sigma_3 \sigma_2)^2)^{\sigma_4^{-1}\sigma_1^2\sigma_3^2\sigma_5^2},\sigma_1^3\sigma_3^3\sigma_5^3,
(\sigma_1^3\sigma_3^3\sigma_5^3)^{\sigma_4\sigma_3 \sigma_2\sigma_4^{-1}})
$$
is a braid monodromy factorization of $\Delta_6^2$ generic at infinity for the Zariski sextic $\bar{\cC}_3$.
Finally, in order to obtain a generic braid monodromy one needs to apply 
Propositions~\ref{prop-gen2} and \ref{prop-gen3} to slightly turn the projection generic and perform Hurwitz 
moves to simplify the braids. One can check that the final generic braid monodromy of $\bar\cC_3$ is:
\begin{equation*}
\sigma_2^3,\sigma_2^{\sigma_3\sigma_4^{-1}},
\sigma_2^{\sigma_3\sigma_4^{-1}\sigma_1\sigma_3\sigma_5},\sigma_3^3,(\sigma_3^3)^{\sigma_4\sigma_5},
\sigma_5^3,\sigma_2^{\sigma_3^{-1}\sigma_4\sigma_1^2\sigma_5^{-1}},(\sigma_5^3)^{\sigma_4},
\sigma_1^3,
(\sigma_1^3)^{\sigma_2},
\sigma_4^3,
(\sigma_5^3)^{\sigma_4\sigma_3 \sigma_2}). 
\end{equation*}

The fundamental group of this curve complement has been extensively studied by Zariski in~\cite{zr:37}. He 
also used deformation arguments to recover the fundamental group studied in~\cite{zr:29}, see \S\ref{subsec-zs},
as well as to study the fundamental group of sextics with six cusps not on a conic. 

Similar arguments can be used to describe the generic braid monodromy of any curve $\bar{\cC}_{n}$, $n$~odd,
which is an irreducible curve of degree $3n$ and $9n$ singularities of type $\ba_{n-1}$ (see~\cite{ji:99}).

\subsection{Ceva arrangement}\label{subsec-ceva}
\mbox{}

We are going to use the classical Ceva arrangement (six lines joining four
points in general position) to find the braid monodromy of the $9$-Ceva arrangement
$$\cC:=\{(x^3-y^3)(y^3-z^3)(x^3-z^3)=0\},$$
and it can be obtained as a Kummer covering of three concurrent lines in a classical Ceva arrangement, 
see \autoref{fig:caso4}.
\begin{figure}[ht]
\begin{center}
\includegraphics[scale=1]{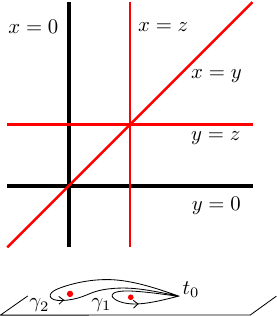}
\end{center}
\caption{Ceva arrangement.}
\label{fig:caso4}
\end{figure}

The extended braid monodromy on $\bb_{2,1}$ of \autoref{fig:caso4} is given by:
$$
\gamma_1\mapsto\sigma_1^2,\quad
\gamma_2\mapsto\sigma_2^2.
$$
Hence, after performing a Kummer cover of order $3$, one obtains:
$$
\begin{matrix}
\gamma_1&\mapsto&\radu_1^2\radu_3^2\radu_5^2\\
\gamma_1^{\gamma_2}&\mapsto&(\radu_1^2\radu_3^2\radu_5^2)^{\radu_4\radu_3\radu_2}\\
\gamma_1^{\gamma_3^2}&\mapsto&(\radu_1^2\radu_3^2\radu_5^2)^{\radu_4^2\radu_3\radu_2^2}\\
\gamma_2^3&\mapsto&\radu_4*(\radu_3\radu_2)^3
\end{matrix}
$$
(see~\eqref{eq-rho3}).

In order to make this monodromy generic at infinity one applies \autoref{prop-gen2} obtaining the new braid:
$$
(\Delta_3^2c^{-1})^3=((\sigma_3\sigma_2)^3)^{\sigma_4^{-1}\sigma_1\sigma_3\sigma_5}.
$$
Using \autoref{prop-def-lines} one can obtain a generic braid monodromy as follows:
\begin{equation*}
\begin{split}
(&\Delta_{5,7}^2,(\Delta_{3,5}^2)^{\Delta_{5,7}^{-1}},(\Delta_{1,3}^2)^{\Delta_{3,5}^{-1}\Delta_{5,7}^{-1}},
(\Delta_{5,7}^2)^{\sigma_4\sigma_3\sigma_2\sigma_7},
(\Delta_{3,5}^2)^{\Delta_{5,7}^{-1}\sigma_4\sigma_3\sigma_2\sigma_7},\\
&(\Delta_{1,3}^2)^{\Delta_{3,5}^{-1}\Delta_{5,7}^{-1}\sigma_4\sigma_3\sigma_2\sigma_7},
(\Delta_{5,7}^2)^{\sigma_4^2\sigma_3\sigma_2^2\sigma_7\sigma_8},
(\Delta_{3,5}^2)^{\Delta_{5,7}^{-1}\sigma_4^2\sigma_3\sigma_2^2\sigma_7\sigma_8},\\
&(\Delta_{1,3}^2)^{\Delta_{3,5}^{-1}\Delta_{5,7}^{-1}\sigma_4^2\sigma_3\sigma_2^2\sigma_7\sigma_8},
\sigma_4*(\sigma_3\sigma_2)^3,((\sigma_3\sigma_2)^3)^{\sigma_4^{-1}\sigma_1\sigma_3\sigma_5}).
\end{split}
\end{equation*}

Another direct application is the monodromy of the MacLane arrangement~\cite{mcl:36}, which is obtained from $\bar\cC_3$ 
by deleting one line. This results in deleting one of the strings, that is,
\begin{equation*}
\begin{split}
(&\Delta_{5,7}^2,(\Delta_{3,5}^2)^{\Delta_{5,7}^{-1}},(\Delta_{1,3}^2)^{\Delta_{3,5}^{-1}\Delta_{5,7}^{-1}},
(\Delta_{5,7}^2)^{\sigma_4\sigma_3\sigma_2\sigma_7},
(\Delta_{3,5}^2)^{\Delta_{5,7}^{-1}\sigma_4\sigma_3\sigma_2\sigma_7},
(\Delta_{1,3}^2)^{\Delta_{3,5}^{-1}\Delta_{5,7}^{-1}\sigma_4\sigma_3\sigma_2\sigma_7},\\
&(\sigma_5^2)^{\sigma_4^2\sigma_3\sigma_2^2},
(\sigma_3^2)^{\sigma_4^2\sigma_3\sigma_2^2},
(\sigma_1^2)^{\sigma_4^2\sigma_3\sigma_2^2},
\sigma_4*(\sigma_3\sigma_2)^3,((\sigma_3\sigma_2)^3)^{\sigma_4^{-1}\sigma_1\sigma_3\sigma_5}).
\end{split}
\end{equation*}
The computational difficulty of the braid monodromy of the $9$-Ceva as well as the MacLane arrangements
comes from the fact that they cannot be given by real equations; MacLane arrangement
is the smallest one with this property. Using our construction everything is
reduced to computing the braid monodromy of a very simple line arrangement.
Braid monodromy for generalized Ceva arrangements~\cite{hirzebruch-geraden} ($n>3$) can also be obtained.

\subsection{A useful nodal cubic.}\label{subsec-useful}
\mbox{}

Let us consider the following cubic $\bar\cC=\{f(x,y,z)=(x+y+z)^3-27xyz=0\}$, whose real picture
appears in \autoref{fig:caso5a}. The usual projection from $P_y$ has four non-generic fibers
$\bar L_t$ at $t=\infty,1,0$, where $\bar L_\infty$ and $\bar L_0$ are tangent lines at inflection
points and $\bar L_1$ is the vertical line through the node (which is real, but whose branches are not).
For this reason the real picture is not enough to recover the braid monodromy of 
$(\bar\cC,P_y,\bar L_\infty)$. However, the extra information required can be obtained from drawing 
the real part of because of the missing complex conjugate branches (shown in \autoref{fig:caso5b}
as dotted lines).
\begin{figure}[ht]
\centering
\subfigure[The nodal cubic.]{
\includegraphics[scale=1]{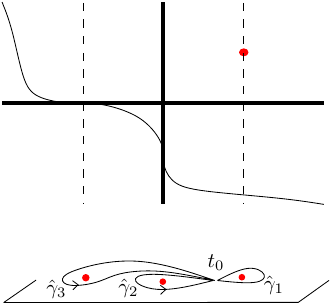}
\label{fig:caso5a}
}
\hfil
\subfigure[Nodal cubic with a \emph{global} real picture.]{
\includegraphics[scale=1]{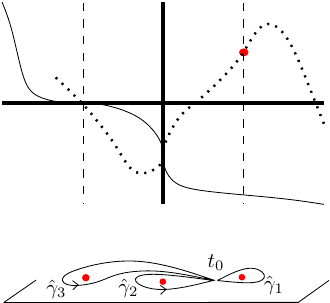}
\label{fig:caso5b}
}
\label{fig:caso5}
\caption{}
\end{figure}

In order to find these dotted lines one can proceed as follows.
Let $p(y):=y^3-a_1 y+a_2 y^2-a_3\in\br [y]$ and assume it has only one real root, say $t_1$. Let
$t_2,t_3\in\mathbb{C}$ be the remaining roots, $t_3=\bar{t}_2$. Note that their common real part is
$\frac{t_2+t_3}{2}$. 

\begin{lema}
\label{lema-real-roots}
Under the above conditions, the polynomial
$$q(y)=y^3-a_1 y^2+\frac{a_1^2+a_2}{4} y-\frac{a_1 a_2-a_3}{8}$$
contains exactly one real root, which is the common real part $\frac{t_2+t_3}{2}$ of the 
complex conjugate roots $t_2$, $t_3$ of~$p(y)$.
\end{lema}

\begin{proof}
Since $a_i=s_i(t_1,t_2,t_3)$ the symmetric polynomial of degree $i$, it is enough to show that
$a_1=s_1(\tilde t_1,\tilde t_2,\tilde t_3)$,
$\frac{a_1^2+a_2}{4}=s_2(\tilde t_1,\tilde t_2,\tilde t_3)$, and
$\frac{a_1 a_2-a_3}{8}=s_3(\tilde t_1,\tilde t_2,\tilde t_3)$,
where $\tilde t_i=\frac{t_j+t_k}{2}$, $\{i,j,k\}=\{1,2,3\}$,
which is a simple exercise.
\end{proof}

Since the affine cubic $f(x_0,y,1)$ satisfies the conditions of \autoref{lema-real-roots} 
for any fixed $\bar x\in \br$ the real parts of the complex roots of $f(x_0,y,1)$ 
are given by the equation:
$$
0\!=y^3+3(x_0+1) y^2+3\frac{4 x_0^2 -x_0 +4}{4}+(x_0+1)\frac{8 x_0^2-65 x_0+8}{8}=
(y+x_0+1)^3-\frac{27}{4} x_0 y-\frac{81}{8} x_0 (x_0+1).
$$
This is enough to compute the required braid monodromy, see~\cite{acct:01} for details.

In order to show the extended braid monodromy $\tilde{\nabla}$ of $\cC$ note that $L_Y$ is 
also tangent to $\cC$ at the inflection point $(-1,0)$. Hence $\tilde B=\{1,0,-1\}$. 
Choose a geometric basis as in \autoref{fig:caso5b} and a system diagram on the fiber given
by decreasing lexicographic order on~$\bc$. For instance, according to \autoref{fig:caso5b} 
$\gs_1$ is the half-twist exchanging the two complex conjugated roots of $f(t_0,y,1)$,
$\gs_2$ is the half-twist exchanging the complex conjugated root of $f(t_0,y,1)$ of negative
real part and $0$. Then, $\tilde{\nabla}:\pi_1(\bc\setminus\tilde{B};t_0)\to\bb_{3,1}$ is 
given by:
$$
\begin{matrix}
\hat{\gamma}_1&\mapsto& \gs_1^2\\
\hat{\gamma}_2&\mapsto& (\sigma_2\sigma_3)^{\gs_1^{-1}\gs_2}\\
\hat{\gamma}_3&\mapsto& \gs_2^6.
\end{matrix}
$$
In order to meet the hypothesis of \S\ref{sec-pencils} one needs the system diagram to be such
that the last string is the one that remains constant. This means one needs to conjugate 
$\tilde{\nabla}$ by $\gs_3^{-1}$. Combining this with a conjugation by $\gs_2^{-1}$, the Hurwitz
move $\hur_2^{-1}$ (changing $\hat{\gamma}_i$ to $\gamma_i$), and another conjugation by $\gs_1^{-1}$ one obtains:

\begin{equation}
\label{eq-bm-nodal}
\begin{matrix}
\gamma_1&\mapsto& \gs_2^2\\
\gamma_2&\mapsto& \gs_2*\sigma_3^6\\
\gamma_3&\mapsto& (\gs_1\gs_2)^{\gs_3^2}.
\end{matrix}
\end{equation}

An easy computation gives that the product of $\Delta^2$ and the inverse of the pseudo-Coxeter
element is $(\gs_1\gs_2)^{\gs_3^2\gs_2}$.

\subsection{The dual of the nodal quartic}\label{subsec-nodal}
\mbox{}

The dual curve of a nodal quartic is a sextic curve with six cusps and four nodes. Moreover, this curve
appears as a generic plane section of the discriminant of the polynomials of degree~$4$ and its fundamental
group has been computed by Zariski~\cite{zr:36}: it is the braid group on $\bs^2$ with four strings.

This sextic can also be obtained from the cubic in \S\ref{subsec-useful} using a Kummer cover $\pi_2$.
Each inflection point in the axes produces two cusps while the double point produces four nodes,
see \autoref{ex-singpin}.

In order to give a braid monodromy for $(\cC_2,P_y,\bar L_\infty)$ from \eqref{eq-bm-nodal} one
needs to combine $\hat{\cub}'_2$ (see~\autoref{ex-cover2}) and \autoref{lema-geometric-basis-cover}.
We start with the first part of \eqref{eq-diag-n}:
\begin{equation}
\array{rccccl}
\bff_5 & \hookrightarrow & \bff_3 & \rightmap{\tilde\nabla} & \bb_{3,1}\\
\tilde\gamma_1 & \mapsto & \gamma_1 & \mapsto & \gs_2^2 \\
\tilde\gamma_2 & \mapsto & \gamma_2 & \mapsto & \gs_2*\gs_3^6 \\
\tilde\gamma_3 & \mapsto & \gamma_1^{\gamma_3} & \mapsto & (\gs_2^2)^{\gs_3^{-2}\gs_1\gs_2\gs_3^2}=
(\gs_2\gs_3^2\gs_1)*\gs_2^2\\
\tilde\gamma_4 & \mapsto & \gamma_2^{\gamma_3} & \mapsto & (\gs_3^6)^{\gs_2^{-1}\gs_3^{-2}\gs_1\gs_2\gs_3^2}
=(\gs_3^6)^{\gs_2^{-1}\gs_1}\\
\tilde\gamma_5 & \mapsto & \gamma_3^2 & \mapsto & \gs_3^{-2}*(\gs_1\gs_2)^2 & \\\endarray
\end{equation}
Using  $\hat{\cub}'_2$ we obtain the braid monodromy of $(\bar{\cC}_2,P_y,\bar{L}_\infty)$:
\begin{equation}
(\recta_2^2\recta_4^2,(\recta_2\recta_4)*\recta_3^3,(\recta_2\recta_4\recta_3\recta_5\recta_1)*(\recta_2^2\recta_4^2),
(\recta_3^3)^{\recta_2^{-1}\recta_4^{-1}\recta_5\recta_1},
((\recta_1\recta_2)^2(\recta_5\recta_4)^2)^{\recta_3}). 
\end{equation}

Using \autoref{prop-gen2}, the monodromy can be made generic at infinity adding
the image by $\hat{\cub}'_2$ of $(\sigma_2^{-1}\sigma_3^2)*(\sigma_2\sigma_1)^3$ 
which is $((\recta_1\recta_2)^2(\recta_5\recta_4)^2)^{\recta_3\recta_2\recta_4}$.

In order to obtain a generic braid monodromy we apply Propositions~\ref{prop-gen3} and~\ref{prop-cusp}.
Table~\ref{tab:c664} shows the decompositions:.

\begin{table}[ht]
 \centering
 \begin{center}
 \begin{tabular}{|c|c|}\hline
$\recta_2^2\recta_4^2$ & $(\sigma_2^2,\sigma_4^2)$\\\hline
$(\recta_2\recta_4)*\recta_3^3$ & $((\recta_2\recta_4)*\recta_3^3)$\\\hline
$(\recta_2\recta_4\recta_3\recta_5\recta_1)*(\recta_2^2\recta_4^2)$  & 
$((\recta_2\recta_4\recta_3\recta_1)*(\recta_2^2),(\recta_2\recta_4\recta_3\recta_5)*(\recta_4^2))$ \\\hline
$(\recta_3^3)^{\recta_2^{-1}\recta_4^{-1}\recta_5\recta_1}$ & $((\recta_3^3)^{\recta_2^{-1}\recta_4^{-1}
\recta_5\recta_1})$\\\hline
$((\recta_1\recta_2)^2(\recta_5\recta_4)^2)^{\recta_3}$ &
$(\recta_2^{\recta_1\recta_3},\recta_1^3,\recta_4^{\recta_5\recta_3},\recta_5^3)
)$\\\hline
$((\recta_1\recta_2)^2(\recta_5\recta_4)^2)^{\recta_3\recta_2\recta_4}$ &
$(\recta_2^{\recta_1\recta_3\recta_2\recta_4},(\recta_1^3)^{\recta_2},
\recta_4^{\recta_5\recta_3\recta_2\recta_4},(\recta_5^3)^{\recta_4}
)$\\\hline
 \end{tabular}

 \end{center}
 \caption{From generic at infinity to generic braid monodromy}
 \label{tab:c664}
\end{table}
We can compute the fundamental group of $\bc^2\setminus\cC_2$. It is well known~\cite{zr:36}
that this group is obtained as a central extension of the braid group of $\bs^2$ with four strings by $\bz$.
More precisely, applying \autoref{thm-zvk4} (without the relation $\mu_6\cdot\ldots\cdot\mu_1=1$)
we obtain a group with generators $\mu_1,\dots,\mu_6$ and relators
\begin{equation*}
\begin{split}
[\mu_4,\mu_5]=[\mu_2^2,\mu_5\mu_4^{-1}]=1,\mu_2\mu_4\mu_2=\mu_4\mu_2\mu_4,\mu_2\mu_5\mu_2=\mu_5\mu_2\mu_5,\\
\mu_6=\mu_2^{\mu_5},\mu_1=\mu_4*\mu_2,\mu_3=\mu_5^{\mu_2\mu_4},\underbrace{1=1}_{7\text{ times}}.
\end{split}
\end{equation*}
This presentation is obtained using \texttt{GAP4}~\cite{GAP4}.

From the original presentation we obtain this one only using Tietze moves of type~I and II. Hence,
by Libgober's result~\cite{li:86}, one can verify that $\bc^2\setminus\cC_2$ has the homotopy type
of $K\vee\bigvee_{j=1}^7\bs^2$, where $K$ is the $2$-complex associated with the presentation
$$
\left\langle x,y,z\left|
[x,z]=[y^2,x z^{-1}]=1,x y x=y x y, y z y=z y z
\right.\right\rangle.
$$

\subsection{Hesse arrangement}
\mbox{}

It can be obtained from the cubic in \S\ref{subsec-useful} using a Kummer cover $\pi_3$: it is the preimage
of the cubic and the three ramification lines. It can be seen using \autoref{prop-gen3}
and \autoref{ex-singpin}, and also directly from the equations:
$$
(x^3+y^3+z^3)^3-27x^3y^3z^3=\prod_{\zeta_1^3=1}(x^3+y^3+z^3-3\zeta_1 x y z).
$$
Hence $x y z=0$ and the above factors give the four reducible members of the pencil of cubics
generated by $x^3+y^3+z^3=0$ and $x y z=0$. The base points of this pencil are the nine common
inflection points of the irreducible cubics of the pencil. The reducible fibers split as product
of lines:
$$
x^3+y^3+z^3-3\zeta_1 x y z=\prod_{\zeta_2^3=1}(x+\zeta_2 y+\zeta_1\bar{\zeta}_2 z).
$$
These are the lines joining the inflection points, i.e., we obtain Hesse arrangement.
We start with the braid monodromy with respect to $(\bar{\cC}_3,P_y,\bar{L}_\infty)$:
$(\sigma_2^2,\sigma_2*\sigma_3^6,(\sigma_1\sigma_2)^{\sigma_3^2})$, and the braid at infinity
is $(\sigma_1\sigma_2)^{\sigma_3^2\sigma_2}$. In order to simplify the braid monodromy
of the Hesse arrangement we conjugate by $\sigma_3^{-2}\sigma_2^{-1}$ and we obtain:
$$
((\sigma_2^2)^{\sigma_3^2},\sigma_3^6,\sigma_2\sigma_1),
\quad \infty\mapsto (\sigma_2\sigma_3^2\sigma_2^{-1}\sigma_3^{-2})*(\sigma_1\sigma_2).
$$
We proceed as in \S\ref{subsec-nodal} but we perform a Hurwitz move to the base of 
\autoref{lema-geometric-basis-cover} in order to obtain simpler braids:
\begin{equation}
\array{rccccl}
\bff_7 & \hookrightarrow & \bff_3 & \rightmap{\tilde\nabla} & \bb_{3,1}\\
\tilde\gamma_1 & \mapsto & \gamma_1 & \mapsto & (\gs_2^2)^{\gs_3^2} \\
\tilde\gamma_2 & \mapsto & \gamma_2 & \mapsto & \gs_3^6 \\
\tilde\gamma_3 & \mapsto & \gamma_1^{\gamma_3} & \mapsto & 
(\gs_2^2)^{\gs_3^2\gs_2\gs_1}\\
\tilde\gamma_4 & \mapsto & \gamma_2^{\gamma_3} & \mapsto & (\gs_3^6)^{\gs_2\gs_1}\\
\tilde\gamma_5 & \mapsto & \gamma_3^3 & \mapsto & (\gs_2\gs_1)^3 & \\
\tilde\gamma_6 & \mapsto & \gamma_3*\gamma_1 & \mapsto & 
(\gs_2\gs_1\gs_3^{-2})*\gs_2^2\\
\tilde\gamma_7 & \mapsto & \gamma_3*\gamma_2 & \mapsto & \gs_2*\gs_3^6.\\\endarray
\end{equation}

Let us recall the map $\tilde{\cub}_3:\bb_{3,1}\to\bb_{9,1}$ of \autoref{prop-motion2}:
\begin{equation}\label{eq-rho39}
\begin{matrix}
\sigma_1&\mapsto&\tilde{\sigma}_{1,1}\tilde{\sigma}_{1,2}\tilde{\sigma}_{1,3}\\
\sigma_2&\mapsto&\tilde{\sigma}_{2,1}\tilde{\sigma}_{2,2}\tilde{\sigma}_{2,3}\\
\sigma_3^2&\mapsto&\tilde{\sigma}^2_{3,3}\tilde{\sigma}_{2,3}\tilde{\sigma}_{1,3}\tilde{\sigma}_{3,2}
\tilde{\sigma}_{2,2}\tilde{\sigma}_{1,2}\tilde{\sigma}_{3,1}\tilde{\sigma}^{-1}_{1,2}\tilde{\sigma}^{-1}_{2,2}
\tilde{\sigma}^{-1}_{1,3}\tilde{\sigma}^{-1}_{2,3}.
\end{matrix}
\end{equation}
For convenience we rewrite this map in the usual generators:
\begin{equation}\label{eq-rho39-1}
\begin{matrix}
\sigma_1&\mapsto&\sigma_{1}\sigma_{4}\sigma_{7}\\
\sigma_2&\mapsto&\sigma_{2}\sigma_{5}\sigma_{8}\\
\sigma_3^2&\mapsto&
\sigma_9^2\sigma_8\sigma_7\sigma_6\sigma_5\sigma_4\sigma_3\sigma_4^{-1}\sigma_5^{-1}\sigma_7^{-1}\sigma_8^{-1}=
(\sigma_9^2\sigma_8\sigma_7)^{\sigma_6\sigma_5\sigma_4\sigma_3\sigma_7\sigma_6}.
\end{matrix} \end{equation}
It is worthwhile to note that $(\sigma_9^2\sigma_8\sigma_7)^3=\Delta_{7,10}^2$.
Let us denote $\cnj:=\sigma_6\sigma_5\sigma_4\sigma_3\sigma_7\sigma_6$
and $\gsiii:=(\sigma_9^2\sigma_8\sigma_7)^{\cnj}$.
Next step is to write down the braid monodromy of $(\bar{\cC}_3,P_y,\bar{L}_\infty)$:
\begin{equation*}
\begin{split}
(&(\sigma_{2}^2\sigma_{5}^2\sigma_{8}^2)^{\gsiii},(\Delta_{7,10}^2)^{\cnj},
(\sigma_{2}^2\sigma_{5}^2\sigma_{8}^2)^{\gsiii\sigma_{2}\sigma_{1}\sigma_{5}\sigma_{4}\sigma_{8}\sigma_{7}},
(\Delta_{7,10}^2)^{\cnj\sigma_{2}\sigma_{1}\sigma_{5}\sigma_{4}\sigma_{8}\sigma_{7}},
\Delta_{1,3}^2\Delta_{4,6}^2\Delta_{7,9}^2,\\
&(\sigma_{2}\sigma_{1}\sigma_{5}\sigma_{4}\sigma_{8}\sigma_{7}\gsiii^{-1})*(\sigma_{2}^2\sigma_{5}^2\sigma_{8}^2),
(\sigma_{2}\sigma_{5}\sigma_{8}\cnj^{-1})*\Delta_{7,10}^2).
\end{split}
\end{equation*}
This monodromy becomes generic at infinity adding
$(\sigma_{2}\sigma_{5}\sigma_{8}\gsiii\sigma_8^{-1}\sigma_5^{-1}\sigma_2^{-1}\gsiii^{-1})*(\Delta_{1,3}^2
\Delta_{4,6}^2\Delta_{7,9}^2)$.
Let us denote $\comm:=\sigma_{2}\sigma_{5}\sigma_{8}\gsiii\sigma_8^{-1}\sigma_5^{-1}\sigma_2^{-1}\gsiii^{-1}$.

To compute a generic braid monodromy to this arrangement one has to apply \autoref{prop-def-lines} 
where we have two vertical lines corresponding to the fifth and eighth braids:
\begin{equation}
\begin{split}
(&(\sigma_{2}^2)^{\gsiii},(\sigma_{5}^2)^{\gsiii},(\sigma_{8}^2)^{\gsiii},(\Delta_{7,10}^2)^{\cnj},
(\sigma_{2}^2)^{\gsiii\sigma_{2}\sigma_{1}\sigma_{5}\sigma_{4}\sigma_{8}\sigma_{7}},
(\sigma_{5}^2)^{\gsiii\sigma_{2}\sigma_{1}\sigma_{5}\sigma_{4}\sigma_{8}\sigma_{7}},
(\sigma_{8}^2)^{\gsiii\sigma_{2}\sigma_{1}\sigma_{5}\sigma_{4}\sigma_{8}\sigma_{7}},\\
&
(\Delta_{7,10}^2)^{\cnj\sigma_{2}\sigma_{1}\sigma_{5}\sigma_{4}\sigma_{8}\sigma_{7}},
\sigma_{10}^2,\sigma_{10}*\Delta_{7,10}^2,(\sigma_{10}\Delta_{7,10})*\Delta_{4,7}^2,
(\sigma_{10}\Delta_{7,10}\Delta_{4,7})*\Delta_{1,4}^2,\\
&(\sigma_{2}\sigma_{1}\sigma_{5}\sigma_{4}\sigma_{8}\sigma_{7}\gsiii^{-1})*\sigma_{2}^2,
(\sigma_{2}\sigma_{1}\sigma_{5}\sigma_{4}\sigma_{8}\sigma_{7}\gsiii^{-1})*\sigma_{5}^2,
(\sigma_{2}\sigma_{1}\sigma_{5}\sigma_{4}\sigma_{8}\sigma_{7}\gsiii^{-1})*\sigma_{8}^2,\\
&(\sigma_{2}\sigma_{5}\sigma_{8}\cnj^{-1})*\Delta_{7,10}^2,
(\sigma_{11}^{-1}\comm)*\sigma_{10}^2,(\sigma_{11}^{-1}\comm\sigma_{10})*\Delta_{7,10}^2,
(\sigma_{11}^{-1}\comm\sigma_{10}\Delta_{7,10})*\Delta_{4,7}^2,\\
&(\sigma_{11}^{-1}\comm\sigma_{10}\Delta_{7,10}\Delta_{4,7})*\Delta_{1,4}^2,\sigma_{11}^2). 
\end{split}
\end{equation}
The braid monodromy was also obtained (computer-aid) in~\cite{ji-Pau}.

\subsection{A smooth cubic and its inflection points and sextics with six cusps outside a conic}
\mbox{}

The Kummer cover $\pi_3$ of the curve $\bar{\cC}$ of equation $(x+y+z)(x+y)(y+z)(x+z)=0$ 
produces a Fermat cubic (as the preimage of the line $x+y+z=0$) while the 
other three lines become the tangent lines to the nine inflection points. 
\begin{figure}[ht]
\begin{center}
\includegraphics[scale=1]{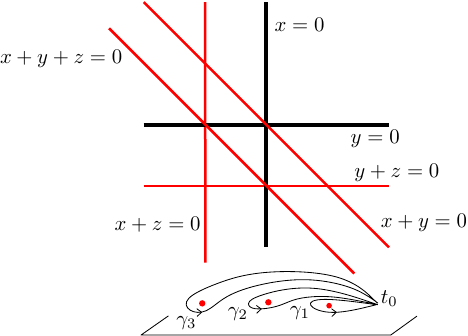}
\end{center}
\caption{Arrangement $\bar{\cC}$.}
\label{fig:casolast}
\end{figure}
The extended braid monodromy of 
$(\bar\cC,P_y,\bar L_\infty)$ is given by
$
(\sigma_2^2,\sigma_2*(\sigma_1^2\sigma_3^2),(\sigma_2\sigma_3\sigma_2^{-1})*\sigma_1^2).
$
In order to meet the hypothesis of \S\ref{sec-pencils} one needs both that the last generator in the 
ordered geometric basis corresponds to $x=0$ and that the system diagram is such that 
the last string is the one that remains constant. One can meet the first requirement by using the 
Hurwitz move $\hur_2^{-1}$:
$$
(\sigma_2^2,(\sigma_2\sigma_3^{-1}\sigma_1^{-1})*\sigma_2^2,\sigma_2*(\sigma_1^2\sigma_3^2)).
$$
In order to switch the string $y=0$ to the last position one must perform
the automorphism $\sigma_i\mapsto\sigma_{4-i}$, but these braids are invariant by the morphism.
After these changes the braid around infinity is $(\sigma_1^2\sigma_3^2)^{\sigma_2}$.

Using \autoref{prop-motion2} one can apply~\eqref{eq-rho39}. 
In~\cite{ea:jag} sextics with six cusps not on a line have been constructed by means
of 3rd Kummer cover where the three lines are tangent lines to a smooth cubic at non-aligned inflection points.
Hence 
we will be interested only in the three tangent lines with these property,
namely, two preimages of $x+z=0$, one preimage of $y+z=0$ and none of $x+y=0$.
Hence, we can forget the second string (recall the automorphism) and we obtain
(the first braid disappears since it is trivial):
$$
((\sigma_2^2)^{\sigma_1},\sigma_1^2)
$$
and hence the braid around infinity is $\sigma_2^2$. We have to forget some strings in the 
mapping~$\hat{\cub}_3$. This causes the morphism to not be well defined in $\bb_{2,1}$, 
but only in the pure braid group:
\begin{equation*}
\begin{matrix}
\sigma_1^2&\mapsto&\sigma_1^2\\
\sigma_2^2&\mapsto&1\\
(\sigma_2^2)^{\sigma_1}&\mapsto&(\sigma_3\sigma_2)^{\sigma_1}.
\end{matrix}
 \end{equation*}
Hence the braid monodromy of the cubic with the tangent lines is given by:
$$
((\sigma_3\sigma_2)^{\sigma_1},(\sigma_3\sigma_2)^{\sigma_1^3},(\sigma_3\sigma_2)^{\sigma_1^5},\sigma_1^6);
$$
where the first two braids correspond to the tangent lines in the curve while the second string
corresponds to the remaining tangent line. Some operations are need: conjugation by 
$\sigma_1^{-3}$, cyclic permutation, permutation of the order of the strings and forgetting one braid 
(at infinity):
$$
(\sigma_3^6,\sigma_3^2*(\sigma_1\sigma_2),\sigma_1\sigma_2),
$$
the braid at infinity being $(\sigma_1\sigma_2)^{\sigma_3^2}$. 
This result coincides with the one obtained using Carmona's program~\cite{car:xx}.

We first consider the braid monodromy factorization induced by this braid monodromy after a 
double cover of the base, that is, $\bff_5 \hookrightarrow \bff_3 \to \bb_{3,1}$:
$$
(\sigma_3^6,\sigma_3^2*(\sigma_1\sigma_2),
(\sigma_3^6)^{\sigma_2},(\sigma_2^{-1}\sigma_3^2)*(\sigma_1\sigma_2),
(\sigma_1\sigma_2)^2,((\sigma_1\sigma_2)^2)^{\sigma_3^2}).
$$
Then perform the morphism $\tilde{\cub}$ (see \autoref{ex-cover2}) and a generic deformation,
resulting in:
$$
(\sigma_3^3,\sigma_1,\sigma_5,\sigma_3^{\sigma_2},\sigma_3^{\sigma_4},(\sigma_3^3)^{\sigma_2\sigma_4},
\sigma_1^{\sigma_2},\sigma_5^{\sigma_4},\sigma_3^{\sigma_2^2\sigma_4},\sigma_3^{\sigma_2\sigma_4^2},
\sigma_2^{\sigma_1},\sigma_1^3,\sigma_4^{\sigma_5},\sigma_5^3,\sigma_2^{\sigma_1\sigma_3},\sigma_1^3,
\sigma_4^{\sigma_5\sigma_3},\sigma_5^3).
$$
A straightforward computation shows that $\pi_1(\bc^2\setminus\cC_2)=\bz$. Using~\cite{li:86}
one can check that 
$$\bc^2\setminus\cC_2\simeq\bs^1\vee\bigvee_{i=1}^{13}\bs^2.$$

\providecommand{\bysame}{\leavevmode\hbox to3em{\hrulefill}\thinspace}
\providecommand{\MR}{\relax\ifhmode\unskip\space\fi MR }
\providecommand{\MRhref}[2]{%
  \href{http://www.ams.org/mathscinet-getitem?mr=#1}{#2}
}
\providecommand{\href}[2]{#2}

\end{document}